\documentclass[11pt]{article}
\usepackage[margin=1.0in]{geometry}
\usepackage[utf8]{inputenc}

\usepackage{amsmath,amsthm}
\usepackage{amsfonts,amssymb,bm}
\usepackage{subfigure}
\usepackage{graphicx}
\usepackage{caption} 
\graphicspath{{Fig/}}
\usepackage{epstopdf}
\usepackage{url}
\usepackage{listings}
\usepackage{xcolor}
\usepackage{hyperref}
\usepackage[capitalize,noabbrev]{cleveref}

\theoremstyle{plain}
\newtheorem{theorem}{Theorem}
\newtheorem{proposition}[theorem]{Proposition}
\newtheorem{lemma}[theorem]{Lemma}

\newtheorem{coro}[theorem]{Corollary}
\newtheorem{property}[theorem]{Property}

\theoremstyle{definition}
\newtheorem{definition}[theorem]{Definition}
\theoremstyle{remark}
\newtheorem{remark}[theorem]{Remark}
\newtheorem{example}[theorem]{Example}

\usepackage{appendix}

\usepackage{booktabs}
\usepackage{stmaryrd} 
\usepackage{authblk}

\usepackage{lineno}
\usepackage{algorithm}
\usepackage{algorithmic}

\newcommand{\keywords}[1]{
  \begin{flushleft}
    \textbf{Keywords:} #1
  \end{flushleft}
  \vspace{1em}
}

\usepackage{xargs}

\usepackage{subfigure}

\usepackage{braket}
\usepackage{multirow}

\usepackage{comment}
\usepackage{tikz,array}
\usetikzlibrary{shapes.geometric}
\usetikzlibrary{shapes.arrows}
\usetikzlibrary{calc}
\usetikzlibrary{intersections}
\usetikzlibrary{positioning}
\usepackage{xcolor}
\DeclareFontFamily{U}{mathx}{\hyphenchar\font45}
\DeclareFontShape{U}{mathx}{m}{n}{<-> mathx10}{}
\DeclareSymbolFont{mathx}{U}{mathx}{m}{n}
\DeclareMathAccent{\widebar}{0}{mathx}{"73}
\newcommand{\mc}[1]{\mathcal{#1}}
\newcommand{\eg}{e.g.\!\,}
\newcommand{\ie}{i.e.\!\,}

\newcommand{\reals}{\mathbb{R}}
\newcommand{\rowof}[2]{#1_{#2,:}}
\newcommand{\colof}[2]{#1_{:,#2}}
\DeclareMathOperator*{\trace}{tr}

\newcommand{\rk}[1][X]{\mathrm{rank}\left(#1\right)}
\DeclareMathOperator{\expm}{Exp}

\newcommand{\rkval}{k}
\newcommand{\rstar}{r^{\star}}
\newcommand{\mstar}{M^{\star}}  
\newcommand{\xstar}{X^{\star}} 
\newcommand{\mrkk}[1][\rkval]{\mathcal{M}_{#1}^{m\times n}} 
\newcommand{\trs}[1][V]{{#1}^{\mathrm{T}}}

\newcommand{\rpd}{\mathrm{RPD}}

\newcommand{\svd}{\mathrm{SVD}}
\newcommand{\dop}{\mathrm{D}}

\newcommand{\ddtt}{\frac{\mathrm{d}^2}{\mathrm{d}t^2}}

\newcommand{\stie}{\mathrm{St}} 
\DeclareMathOperator*{\grad}{\mathrm{grad}}
\DeclareMathOperator*{\hess}{\mathrm{Hess}\!}

\newcommand{\mleqk}[1][\rkval]{\mathcal{M}_{\leq#1}} 
\newcommand{\mtotal}[1][\rkval]{\widebar{\mathcal{M}}_{#1}^{m,n}} 
\newcommand{\glin}[1][\rkval]{\mathrm{GL}\left(#1\right)} 
\newcommand{\vecsp}[1][m\times n]{\reals^{#1}}
\newcommand{\prodspstar}[1][\rkval]{\reals^{m\times#1}_*\times\reals^{n\times#1}_*} 
\newcommand{\prodsp}[1][\rkval]{\reals^{m\times#1}\times\reals^{n\times#1}}

\newcommand{\thorsp}[1][x]{\mathcal{H}_{\bar{#1}} }
\newcommand{\tversp}[1][x]{\mathcal{V}_{\bar{#1}} }
\newcommand{\ttansp}[1][\bar{x}]{T_{#1}\mtotal}
\newcommand{\ttanv}[1][\xi]{\widebar{#1}}
\newcommand{\tansp}[1][X]{T_{#1}}

\newcommand{\frobb}[1][X]{\|#1\|_{\mathrm{F}}^2}
\newcommand{\frob}[1][X]{\|#1\|_\mathrm{F}}
\newcommand{\fro}[1]{\|#1\|_\mathrm{F}}
\newcommand{\rinfnorm}[1][X]{\|#1\|_{2,\infty}}

\newcommand{\nop}[1]{\|#1\|_{\mathrm{op}}}
\newcommand{\tfpartg}[1][\bar{x}]{\partial_{G}\bar{f}\left(#1\right)}
\newcommand{\tfparth}[1][\bar{x}]{\partial_{H}\bar{f}\left(#1\right)}
\newcommand{\rgrad}[1][X]{\mathrm{grad}f\left(#1\right)}

\newcommandx{\trgrad}[2][1=x,2=]{\mathrm{grad} \bar{f}\left(\bar{#1}_{#2}\right)}
\newcommandx{\tmet}[2][1=x,2=]{\bar{g}_{\bar{#1}_{#2}}}
\newcommandx{\rmet}[2][1=X,2=]{g_{#1_{#2}}}
\newcommand{\ttanvl}[1][\xi]{\bar{#1}^{(1)}}
\newcommand{\ttanvr}[1][\xi]{(\bar{#1}^{(2)})^{\mathrm{T}}}
\DeclareMathOperator{\proj}{P}
\DeclareMathOperator{\id}{I}
\newcommandx{\matsp}[2][1=m,2=n]{\reals^{#1\times #2}} 
\newcommand{\tvec}[1][X]{\mathrm{vec}\left(#1\right)}
\newcommand{\iset}[1]{[\![#1]\!]}

\newcommand{\usvsp}{\stie(m,\rkval)\times \mathcal{S}_{++}(\rkval)\times \stie(n,\rkval)}
\newcommand{\uysp}{\stie(m,\rkval)\times\reals_*^{n\times \rkval}}

\newcommand{\bstar}{B^{\star}}
\newcommand{\brip}{\beta}
\DeclareMathOperator{\opa}{\mathcal{A}}
\newcommand{\duop}[1][A]{{#1}^{*}}
\newcommand{\po}[1][\Omega]{P_{#1}}
\newcommand{\pom}[1][\Omega]{P_{#1}}
\newcommand{\pu}[1][U]{{P}_{#1}} %{\mathcal{P}_{#1}}
\newcommand{\opvx}[1][X]{\mathcal{G}_{#1}}

\newcommand{\ssz}{\theta} % either \ssz or \theta are used for stepsize 
\newcommand{\dttx}[1][]{\Delta_{#1}}  

\newcommand{\algseq}{\{\bar{x}_{t}\}_{t\geq 0}}
\newcommand{\indseq}{\{X_{t}\}_{t\geq 0}}
\newcommand{\sigmax}{\sigma_{\max}^{\star}}
\newcommand{\sigmin}{\sigma_{\rkval}^{\star}}

\newcommand{\dnm}{\delta} % 
\newcommand{\nmb}[1][\dnm]{\mathcal{B}^{\star}(#1)}  % 
\newcommand{\muu}{\mu_{f}} % 
\newcommand{\cbd}{(1-\brip)\frac{\sigmin}{\lip}} 
\newcommand{\rlip}{L_g}
\newcommand{\lip}{L_f}

\newcommand{\cca}{{\bar{C}_{1}}} 
\newcommand{\tcca}{{\bar{C}_{2}}} 

\newcommand{\ccb}{{\bar{C}_{3}}} 
\newcommand{\ccc}{{\bar{C}}} %{C_{k}}
\newcommand{\nmg}[1][\mu]{\mathcal{B}^{\star}_{\delta_0}(#1)}
\newcommand{\brk}{\bar{r}_{\rkval}}
\newcommand{\crinf}{\epsilon_{\theta}}

\newcommand{\rev}[1]{#1}
\newcommand{\revv}[1]{#1}
\newcommand{\revj}[1]{#1}

\newcommand{\bleu}[1]{{#1}}
\newcommand{\redl}[1]{{#1}}

\newcommand{\checkk}[1]{}

\title{On the Analysis of Optimization with Fixed-Rank Matrices: a Quotient Geometric View
} 

\author{Shuyu Dong\thanks{LISN, INRIA, Universit\'e Paris-Saclay, Gif-sur-Yvette, France. (shuyu.dong@inria.fr)} \qquad 
Bin Gao\thanks{LSEC, AMSS, Chinese Academy of Sciences. (gaobin@lsec.cc.ac.cn)}
\qquad Wen Huang\thanks{School of Mathematical Sciences, Xiamen University, Fujian, China. (wen.huang@xmu.edu.cn)}
\qquad Kyle A. Gallivan\thanks{Department of Mathematics, Florida State University, Tallahassee, USA. (kgallivan@fsu.edu)} 
}

\date{} 
\begin{document}
\maketitle

\begin{abstract} 
We study a type of Riemannian gradient descent (RGD) algorithm,
designed through Riemannian preconditioning, for optimization on $\mrkk$---the set of $m\times n$ real matrices with a fixed rank $k$. 
Our analysis is based on a quotient geometric view of $\mrkk$: 
by identifying this set with the quotient manifold of a two-term product space $\prodspstar$ of matrices with full column rank via matrix factorization,
we find an explicit form for the update rule of the RGD algorithm, which leads to a novel approach to analysing their convergence behavior in rank-constrained optimization. 
We then deduce some interesting properties that reflect how RGD distinguishes from other matrix factorization algorithms such as those based on the Euclidean geometry. 
In particular, we show that the RGD algorithm is not only faster than
Euclidean gradient descent but also does not rely on balancing techniques to
ensure its efficiency while the latter does. 
We further show that this RGD algorithm
is guaranteed to solve matrix sensing and matrix completion problems with
linear convergence rate under %
the restricted positive definiteness property. % 
Numerical experiments on matrix sensing and completion are provided to demonstrate these properties.  
\end{abstract}

\keywords{
fixed-rank matrices, Riemannian optimization, quotient manifold, gradient descent methods, nonconvex optimization, matrix recovery 
}

%%-------main text 
% \input{maincontent.tex}
\section{Introduction}\label{sec:intro}

Optimization with low-rank matrices is a fundamental problem that arises in signal processing, machine learning and computer vision. % 
In principal component analysis, matrix recovery and data clustering for example, the most meaningful information in the data is structured, such that it can be captured by a matrix with an intrinsically low rank~\cite{udell2019big}. 
Therefore, low-rank matrix models enjoy the advantage of having a low complexity without compromising the accuracy or representativity. 
One approach to optimizing low-rank matrix models is by using the matrix nuclear norm~\cite{Candes2008b,recht2010guaranteed,Candes2010}, which is a convex relaxation of the matrix rank. %
Another approach is to represent an $m\times n$ matrix $X$ through low-rank matrix factorization such as $X=A\trs[B]$, where $A$ and $B$ are thin factor matrices of size $m\times \rkval$ and $n\times\rkval$ respectively. 
Although the matrix factorization approach induces a nonconvex optimization problem, it presents several advantages over the previous convex relaxation approach due to much lower cost in memory and computation. 
Algorithms for low-rank matrix factorization can be regrouped into two main types, alternating minimization~\cite{zhou2008large,jain2013low,Hardt2014a} and gradient descent algorithms~\cite{Keshavan2009,keshavan2010matrix,park2018finding}. 
Recent advances in matrix recovery problems such as compressed sensing and matrix completion~\cite{Sun2016a,ge2016matrix,tu2016low,Zheng2016} %
shed light on the absence of spurious local minima in %
these problems under mild conditions, and thus they explain formally the success of (Euclidean) gradient descent algorithms for matrix recovery, despite the nonconvexity of matrix factorization. 

Riemannian algorithms, in a similar spirit as Euclidean gradient descent but exploiting non-Euclidean geometries on low-rank matrix space, have % 
shown performances superior to algorithms using the Euclidean geometry. 
As an important example of low-rank matrix spaces, the set $\mrkk$ of fixed-rank matrices % 
\begin{equation}\label{def:mk}
\mrkk=\left\{X\in\vecsp: \rk[X] = k\right\}
\end{equation}
is a smooth Riemannian manifold of dimension $(m+n-k)k$ and can be characterized through the factorization of rank-$k$ matrices (\eg,~\cite{Vandereycken2013b}). % 
\Cref{tab:mf} shows some basic information of different matrix factorization approaches in the related work, alongside the convex relaxation approach. %
For example, Vandereycken~\cite{Vandereycken2013b} used the rank-$\rkval$ singular value decomposition ($\svd$) to identify a matrix $X\in\mrkk$ with the point $(U,\Sigma,V)$, where $U,\Sigma$ and $V$ are the factor matrices of the rank-$\rkval$ $\svd$ of $X$, and considered $\mrkk$ as an embedded Riemannian submanifold of $\reals^{m\times n}$ with the Euclidean metric. 
Wei et~al.~\cite{wei2016guarantees,tanner2016low} proposed
several variants of the iterative hard thresholding (IHT) % 
algorithm, also based on the embedded Riemannian manifold structure of $\mrkk$ for matrix recovery problems with the fixed-rank constraint, and provided exact-recovery guarantees of the algorithm for compressed sensing and matrix completion~\cite{wei2020guarantees} using the restricted isometry property
(RIP). % 

\begin{table}[!htbp]
%\small
\centering 
\caption{Different matrix factorization approaches alongside convex relaxation of the matrix rank.}
\label{tab:mf}
\begin{tabular}{l|cc|c}
\hline
                        & Search space   & \# parameters   & Reference                                                           \\ \hline
Conv. relax.            & $\vecsp$       & $mn$            & \cite{Candes2008b,recht2010guaranteed,Candes2010}                   \\ 
$X=U\Sigma\trs[V]$      & $\usvsp$       & $(m+n+1)\rkval$ & \cite{Vandereycken2013b,Mishra2014}                                  \\ 
$X=U\trs[Y]$            & $\uysp$        & $(m+n)\rkval$   & \cite{Boumal}                                                         \\ 
% $X=G\trs[H]$            & $\prodsp$     & $(m+n)\rkval$   &   \cite{tong2021accelerating}              \\ 
$X=G\trs[H]$            & $\prodspstar$ ($\prodsp$) & $(m+n)\rkval$   & ours,~\cite{meyer2011linear,mishra2012,wei2016guarantees,tanner2016low}  (\cite{tong2021accelerating})             \\ 
\hline
\end{tabular}
\end{table}

Apart from the embedded manifold view, it is known that $\mrkk$~\eqref{def:mk} can be understood as a quotient space: consider a product space $\mtotal:=\prodspstar$ of matrices with full column rank $\rkval$, the projection 
\[\pi: \prodspstar\to\mrkk: (G,H)\mapsto G\trs[H]\]
induces the following equivalence relation $\sim$: 
\[(G,H) \sim (G',H') \quad\text{if and only if}\quad G\trs[H] = G'\trs[H'].\] %
Since the equivalent classes of $\sim$ are the fibers of $\pi$ and $\mrkk\subset\vecsp$ is the image of $\pi$, the mapping $\pi$ induces an one-to-one correspondence between $\mrkk$ and $\prodspstar/\sim$, hence the quotient structure of $\mrkk$; see \Cref{fig:quot-pic}. 
\begin{figure}[!h]
  \centering
\begin{tikzpicture}
   \node (E) at (1,0) {\hspace{-5.6em}$\mtotal=\mathbb{R}_*^{m\times k}\times \mathbb{R}_*^{n\times k}$};
   \node[below=of E] (M) {$\mathbb{R}_*^{m\times k}\times \mathbb{R}_*^{n\times k}/\sim$};
   \node[right=of M] (N) {$\mrkk$};
   \draw[<->] (M)--(N) node [midway,below] {~};
   \draw[->] (E)--(M) node [midway,left] {$\Pi$};
   \draw[->] (E)--(N) node [midway,above] {$\pi$};
   \node (F) at (6.5,0) {$\mtotal$};
   \node[below=of F] (G) {$\mrkk$};
   \node[right=of G] (H) {$\reals$};
   \draw[->] (G)--(H) node [midway,below] {$f$};
   \draw[->] (F)--(G) node [midway,left] {$\pi$};
   \draw[->] (F)--(H) node [midway,right] {~$\bar{f}=f\circ\pi$};
  \end{tikzpicture}
 \caption{Relations between the product space $\prodspstar$ and the quotient space $\mrkk$.\label{fig:quot-pic}}
\end{figure}
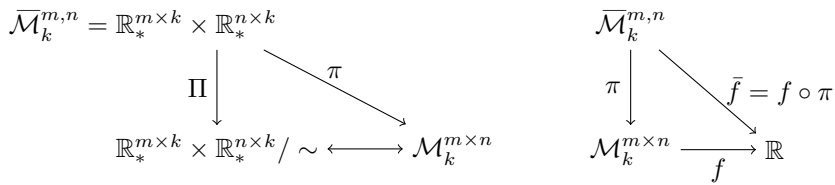 % 
Subsequently, the optimization of a real-valued function $f$ on $\mrkk$ can be seen as the following matrix factorization optimization: 
\begin{equation}\label{prog:main-generic}
    \min_{(G,H)\in\prodspstar} \bar{f}(G,H):= f\circ \pi(G,H). 
\end{equation} 
In view of this quotient strucutre, Mishra et al.~\cite{mishra2012} proposed a
Riemannian gradient descent (RGD) algorithm using two-term matrix factorization
and a preconditioning technique~\cite{mishra2016riemannian} specially adapted to the squares loss function. 
Tong et al.~\cite{tong2021accelerating} analyzed the convergence property of
the {\it preconditioned RGD algorithm} of \cite{mishra2012} on the product space $\prodsp$ (under the name
of ScaledGD) by leveraging a distance on  $\prodsp$ that is specifically
invariant on fixed-rank manifolds.  
Boumal et al.~\cite{Boumal} used the factorization $X=UY$, where $U\in\reals^{m\times \rkval}_*$ is considered as a point of the Grassmannian manifold (a quotient space of the Stiefel manifold $\stie(m,\rkval)$), and proposed a variable projection method using a Riemannian trust-region algorithm for matrix completion with fixed-rank matrices. %
We refer to \Cref{tab:mf} and~\cite{Mishra2014,absil2014two} for a thorough overview. %
More recently, Huang et al.~\cite{huang2018blind} proposed an efficient Riemannian gradient descent algorithm for blind deconvolution based on the quotient manifold structure of the set of rank-$1$ (complex-valued) matrices with guaranteed convergence using the RIP; 
Luo et al.~\cite{luo2021recursive} proposed a new sketching algorithm on the set of fixed-rank matrices and proved that the algorithm enjoys high-order convergence for low-rank matrix trace regression and phase retrieval under similar conditions. % 

In this paper, we are interested in how Riemannian algorithms behave in solving problem~\eqref{prog:main-generic}, and how their behavior is related to the landscape of the original problem $\min_{X\in\mrkk} f(X)$. %
To answer this question, we start with two noticeable difficulties underlying~\eqref{prog:main-generic}: 
one difficulty is the nonconvexity of~\eqref{prog:main-generic}, which is inherited from the nonconvex space $\mrkk$ via the mapping $\pi$. %
Another difficulty is that $\pi$ being not injective, there is no unique---and in fact an infinite number
of---matrix representations $(G,H)$ in $\prodspstar$ for each $X\in\mrkk$, and the
performances of algorithms such as Euclidean gradient descent generally vary according to the actual locations of
the iterates (such as the initial point) inside their respective equivalence classes on $\prodspstar$. % 
The non-uniqueness of matrix representations also leads to the difficulty that the minima of the
problem~\eqref{prog:main-generic} are degenerate~\cite{absil2014two}. %
\begin{figure}[htpb]
\centering
\includegraphics[width=0.40\textwidth]{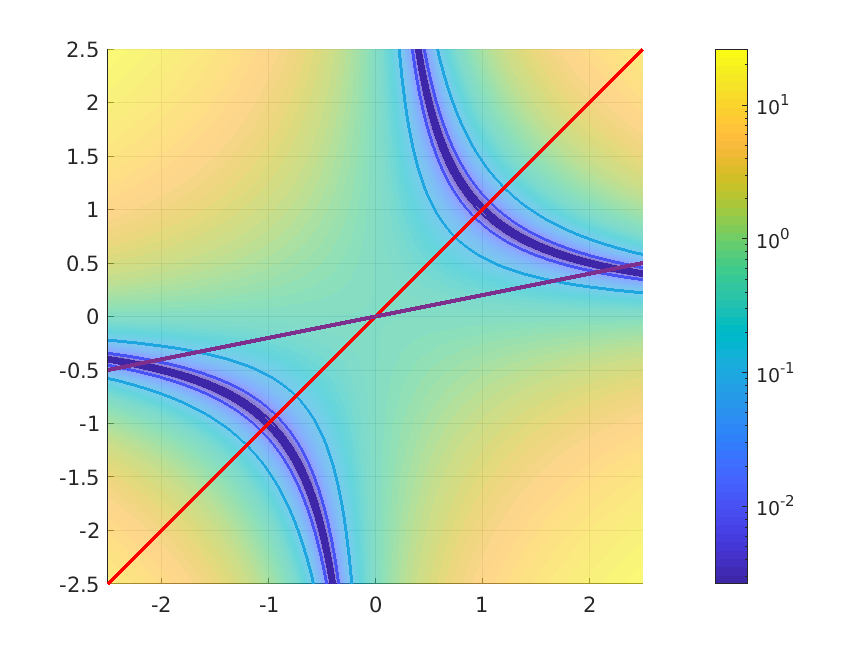}\qquad
\includegraphics[width=0.40\textwidth]{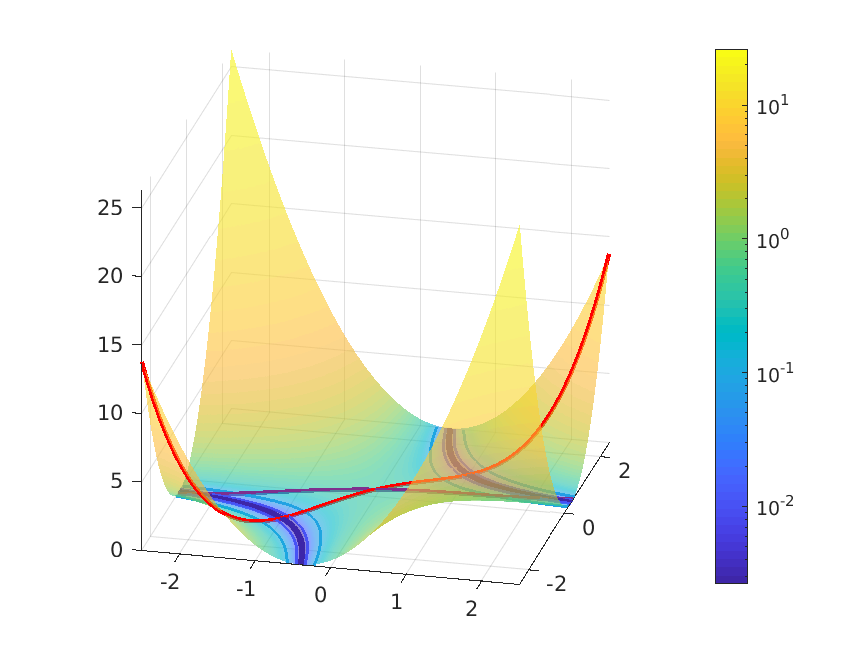}
\caption{Landscape of the cost function of the simplest matrix factorization, on the 2D plane.}
\label{fig:lands-mf}
\end{figure}
These difficulties can be visualized in the landscape of even the simplest matrix factorization problem: \Cref{fig:lands-mf} shows the graph of $\bar{f}:\reals_*\times\reals_*\to\reals: {(x,y)}\mapsto \frac{1}{2}({xy}-A)^2$, for $A=1$, along with two different paths (in red and purple) passing through the equivalence
classes on $\reals_*\times\reals_*$. % 
Concretely, we investigate the properties of the RGD algorithm of~\cite{mishra2012} through the lens of quotient geometry of $\mrkk$. 
We present new results leading to the confirmation that this algorithm bypasses the above-mentioned difficulties of~\eqref{prog:main-generic}. %
Our results about the RGD algorithm~\cite{mishra2012} also show substantial differences between algorithms designed with an explicit metric on the quotient manifold $\mrkk$ and algorithms that are not based on the quotient geometry of $\mrkk$, such as the Riemannian algorithms using metric projection~\cite{Vandereycken2013b,wei2016guarantees,tanner2016low,wei2020guarantees} and algorithms based on the Euclidean geometry~\cite{Sun2016a}. % 
More precisely, the main contributions are as follows. 

\subsection{Contributions} 
First, we prove that the RGD algorithm of~\cite{mishra2012} (on $\prodspstar$) induces a sequence on $\mrkk$ that admits an explicit update rule, and that
the sequence is invariant to changes of its iterate $(G,H)\in\prodspstar$ in the equivalence classes. 
Second, we analyse the convergence behavior of the algorithm under the {\it restricted positive definiteness}\/ (RPD) property~\cite{10.1093/imanum/drz061} in solving the following low-rank matrix optimization problems: %
\begin{example}[Compressed sensing]\label{exa:lrqp-ms}
   Let $\mstar$ be an $m\times n$ matrix and $b^{\star}= \Phi(\mstar)$ a vector observed through a matrix sensing operator,
   $\Phi:\vecsp\to\reals^{d}$.
   Recover $\mstar$ by minimizing 
       $f(X) := \frac{1}{2}\|\Phi(X)-b^{\star}\|_{2}^2$ 
   with a low-rank matrix~$X$. 
\end{example}
\begin{example}[Matrix completion]\label{exa:opt-mk-via-mtotal}
    Let $\mstar\in\vecsp$ be a matrix observed only on an index set $\Omega\subset \iset{m}\times \iset{n}$. 
Complete $\mstar$ by minimizing $f(X):=\frac{1}{2}\fro{\po(X-\mstar)}^2$ with a low-rank matrix~$X$. 
\end{example}
Through the RPD property of the objective function $f$ around the 
true hidden matrix $\mstar$, we demonstrate the existence of a region of attraction on $\mrkk$ %
in which the landscape of $f$ in the ambient space $\vecsp$ is preserved. 
Then, we prove results about the local convergence rate of the quotient
geometry-based RGD algorithm %
for the minimization of $f$ on $\mrkk$. 
To the best of our knowledge, this is the first convergence rate analysis 
of the RGD algorithm of~\cite{mishra2012} using quotient geometry of the fixed-rank matrix manifold. 

We illustrate the main results through numerical experiments: we show that the quotient
manifold-based algorithms not only enjoy the benefits of low-rank
matrix factorization but also present desirable {\it invariance}\/ properties
that the Euclidean gradient descent does not possess. In particular, 
its convergence behavior % 
does not vary with changes in the {\it balancing}\/ between the factor
matrices, 
while the performance of Euclidean gradient descent can be easily
deteriorated due to unbalanced factor matrices.

In summary: (i) we show formally that the Riemannian gradient descent algorithm under the aforementioned quotient geometric setting enjoys desirable invariance properties. % 
A new result (\Cref{lemm:rgd-mtotal2manf}) about this quotient manifold-based
RGD algorithm is given; (ii) for a class of low-rank matrix optimization
problems, we provide new results about the geometric properties of the RGD
algorithm around critical points (\Cref{lemm:f-rsc-riem}) under the restricted
positive definiteness property, about sufficient descent conditions for
RGD (\Cref{lemm:4.8z} and \Cref{coro:4.8z}), and about the local linear convergence of RGD (\Cref{lemm:f-rsc-fort}) on $\mrkk$. 
\subsection{Preliminaries}
\label{subsec:notation-prelims}
Given the integers $m,n\geq 1$ and $k\leq\min(m,n)$, 
\begin{equation}\label{def:mtotal}
    \mtotal=\prodspstar 
\end{equation}
denotes the product space of $m\times k$ and $n\times k$ real matrices with full column-rank. %
The set of $m\times k$ orthonormal real matrices, \ie, the Stiefel manifold, is
denoted by $\stie(m,k)$. The set of $k\times k$ positive definite matrices is
denoted by $\mathcal{S}_{++}(k)$. 
 
A point in $\mtotal$~\eqref{def:mtotal} is denote by $\bar{x}=(G_{\bar{x}},H_{\bar{x}})$, $(G,H)$ or simply $\bar{x}$ indifferently. By default, the symbols $G$ and $H$ signify the % 
$m\times\rkval$ and $n\times\rkval$ matrices of $\bar{x}$ respectively; they also constitute a pair of left and right factor matrices of $X=G\trs[H]$. 
For any $\bar{x}\in\mtotal$, the tangent space to $\mtotal$ at $\bar{x}$ is $\tansp[\bar{x}]\mtotal=\prodsp$. A tangent vector $\ttanv\in\tansp[\bar{x}]\mtotal$ is denoted as $\ttanv=(\ttanvl,\bar{\xi}^{(2)})$. 
For any $Y\in\reals^{m\times k}_*$, the Euclidean metric on the tangent space $\tansp[Y]\reals^{m\times k}_*\simeq\reals^{m\times k}$ is defined and denoted by $\braket{V,W}:=\trace(\trs[V]W)$, for $V,W\in\tansp[Y]\reals^{m\times k}_*$, % 
where $\braket{\cdot,\cdot}$ is also called the Frobenius inner product. %

The index set $\{1,\dots, n\}$ is denoted as $\iset{n}$. For any $X\in\mrkk$, the $k$-th singular value of $X$, \ie, the minimal non-zero singular value, is denoted as $\sigmin(X)$. 
For brevity, we denote the maximal and the minimal non-zero singular values of a given constant matrix $\mstar$ by $\sigmax$ and $\sigmin$, respectively.
The row vector from the $i$-th row of a matrix $X$ is denoted as $\rowof{X}{i}$ and the $2$-norm of the matrix row as $\|\rowof{X}{i}\|_2$. The maximal row norm of a matrix $X$ with $m$ rows is denoted and defined as $\rinfnorm[X]:=\max_{1\leq i\leq m}\|\rowof{X}{i}\|_2$. % 
The spectral norm of a symmetric positive semidefinite matrix $A$ and the operator norm of the linear operator $A: X\mapsto AX$ are denoted by $\|A\|_{2}$. We denote the adjoint of a linear operator $A$ by $A^*$. 

\paragraph{The RPD property.}

Due to the rank constraint of~\eqref{prog:main-mk}, the optimization of $f$~\eqref{eq:def-f-lrqp} on $\mrkk$ is nonconvex. However, some properties of the quadratic function $f$ are preserved %
on rank-constrained matrix spaces. % 
The following restricted positive definiteness (RPD) property~\cite{10.1093/imanum/drz061} characterizes the well-conditionedness of $f$ on the set of rank-constrained matrices. % 

\begin{definition}[{RPD property~\cite[Definition 3.1]{10.1093/imanum/drz061}}]
    \label{def:lrqp-rip}
    For an integer $r\geq 1$ and a parameter $0\leq\brip_r <1$, the operator $\opa$ of $f$~\eqref{eq:def-f-lrqp} satisfies a \emph{$(\brip_r,r)$-$\rpd$ property} if 
\begin{equation}
    \label{eq:def-rip}
    (1-\brip_r)\frob[Z]^2 \leq \braket{Z,\opa(Z)} \leq (1+\brip_r)\frob[Z]^2,
\end{equation}
for any $Z\in\mleqk[r]$. The smallest $\brip_r\geq 0$ for the $(\brip_r,r)$-$\rpd$ property to hold is called the \emph{$\rpd$ constant} of $\opa$ on $\mleqk[r]$. 
\end{definition}

The $\rpd$ property is equivalent to the restricted isometry property (RIP) condition~\cite{candes2005decoding,candes2006stable} in the literature of compressed sensing. This property can be satisfied with overwhelmingly high probability for a large family of random measurement matrices, for example, the normalized
Gaussian and Bernoulli matrices~\cite{5730578,recht2010guaranteed,wei2016guarantees}. 
For matrix completion, the $\rpd$ property also holds under reasonable assumptions on the hidden true matrix $\mstar$ and the sampling distribution~\cite{candes2009exact}.

\subsection{Organization} 
The rest of this paper is organized as follows. 
In Section~\ref{sec:opt-on-mk}, we present algorithms based on a Riemannian preconditioning technique on $\mtotal$ and then the quotient manifold structure of
$\mrkk$ related to these algorithms. %
In Sections~\ref{sec-ropt-mk:lrqp}, we propose the main results about the RGD algorithm of Section~\ref{sec:opt-on-mk}. In Section~\ref{ssec-ropt-mk:lrqp-intro}, we present a new convergence analysis of the RGD with performance guarantees for matrix recovery problems on $\mrkk$. % 
Numerical experiments and results are presented in Section~\ref{sec-mtns20:exp}. 
We conclude the paper in Section~\ref{sec-throptmk:conclusion}.

\section{Algorithms on $\mtotal$ and quotient geometry of \texorpdfstring{$\mrkk$}{Mk}}
\label{sec:opt-on-mk}

In this section, we revisit two Riemannian gradient-based algorithms, which are designed for solving~\eqref{prog:main-generic} through Riemannian preconditioning. The preconditioning technique induces a Riemannian metric on $\mtotal$, which turns out to be invariant on the equivalence classes of the matrix product mapping $\pi$ in the total space $\mtotal$~\eqref{def:mtotal}. This invariance property is essential to the quotient structure of $\mrkk$ and the properties of the RGD algorithm, as we will present in Section~\ref{ssec:geom-manf-as-quosp} and later in Section~\ref{sec-ropt-mk:lrqp}. 

\subsection{Gradient descent on~\texorpdfstring{$\mtotal$}{the total space} by Riemannian preconditioning}\label{ssec:opt-manf-via-mtotal}

First, we describe a non-Euclidean metric on the product space $\mtotal$~\eqref{def:mtotal} used in a Riemannian algorithm for low-rank matrix completion~\cite{mishra2012}. 

\begin{definition}[\cite{mishra2012}]\label{def:cropt-mk-metric-precon}
Given $\bar{x}:=(G,H)\in\mtotal$, let $\bar{g}_{\bar{x}}:\tansp[\bar{x}]\mtotal\times\tansp[\bar{x}]\mtotal$ denote an inner product defined as follows, 
\begin{equation}\label{eq-throptmk:def-metric-precon}
    \bar{g}_{\bar{x}}(\ttanv,\ttanv[\eta]) =
    \trace(\bar{\xi}^{(1)T}\bar{\eta}^{(1)} (\trs[H]H)) + \trace(\ttanvr
    \bar{\eta}^{(2)} (\trs[G]G)),
\end{equation}
for $\ttanv,\ttanv[\eta]\in\tansp[\bar{x}]\mtotal$. %
\end{definition}
It can be shown that $\bar{g}_{\bar{x}}(\cdot,\cdot)$ is a Riemannian metric
since it is symmetric and positive-definite at any $\bar{x}\in\mtotal$ and it is a smooth-varying bilinear form on $\mtotal$. 
By definition, %
the Riemannian gradient of a function $\bar{f}$ at $\bar{x}\in\mtotal$ is the unique vector, denoted as $\trgrad\in\ttansp$, such that $\tmet[x](\bar{\xi},\trgrad) = \dop\bar{f}(\bar{x})[\ttanv]$, $\forall\ttanv\in\ttansp$. 
Therefore, given the metric~\eqref{eq-throptmk:def-metric-precon}, the Riemannian gradient of $\bar{f}$ has the following form,
\begin{equation}\label{eq-throptmk:rgrad-man-precon}
    \trgrad = \big(\tfpartg[\bar{x}]{{(\trs[H] H)}^{-1}},
    \tfparth[\bar{x}]{{(\trs[G] G)}^{-1}}\big),
\end{equation}
where $\tfpartg[\bar{x}]$ and $\tfparth[\bar{x}]$ are the two partial differentials of $\bar{f}$.

The metric $\bar{g}$ defined in~\eqref{eq-throptmk:def-metric-precon} was
proposed by Mishra et al.~\cite{mishra2012} for matrix completion (\Cref{exa:opt-mk-via-mtotal}) using fixed-rank matrices, and can be seen as a metric deduced from the Riemannian preconditioning technique~\cite{mishra2016riemannian} in the context of two-term matrix factorization with a squares-loss function $\bar{f}(G,H) = \frac{1}{2} \| G\trs[H] - M^\star \|_F^2$. 
Indeed, % 
one can see that the Riemannian gradient~\eqref{eq-throptmk:rgrad-man-precon} %
is in fact the solution to the secant equation in $\ttanv$: $\mathcal{H}(\ttanv) = (\partial_G\bar{f}(\bar{x}), \partial_H\bar{f}(\bar{x}))$, 
where $\mathcal{H}:\tansp[\bar{x}]\mtotal\to\tansp[\bar{x}]\mtotal$ is %
defined as 
\[
\mathcal{H}(\ttanv):= \begin{pmatrix} \partial_{G}^{2}\bar{f}(G,H) & 0 \\
                               0 &   \partial_{H}^{2}\bar{f}(G,H) 
             \end{pmatrix}(\ttanv) = (\ttanvl \trs[H] H,  \ttanv^{(2)} \trs[G] G).\] %
Note that $\mathcal{H}(\ttanv)\approx \dop^2 \bar{f}(G,H)[\ttanv]$. Therefore, the Riemannian gradient~\eqref{eq-throptmk:rgrad-man-precon} is an approximation of the Newton direction of $\bar{f}$ at $(G,H)$. 
The metric $\bar{g}$~\eqref{eq-throptmk:def-metric-precon} is also referred to as the {\it preconditioned}\/ metric on $\prodspstar$. 

\paragraph{Riemannian gradient descent (RGD)}\label{ssec:alg-details} %

For a matrix factorization problem in the form of~\eqref{prog:main-generic}, we
present the RGD as in \cref{algo-throptmk:generic-rgd}. The search direction is set to be the negative Riemannian gradient~\eqref{eq-throptmk:rgrad-man-precon}. 
The operation needed for the gradient descent update step (line~\ref{algl-throptmk:rgd-update}) on $\mtotal$ is chosen
to be the identity map, which is a valid retraction operator on the $\mtotal$. 
One can find the use of the identity map as a retraction in a total space:
$\mtotal$~\eqref{def:mtotal} and $\mathbb{C}^{K}_{*} \times \mathbb{C}^{N}_{*}$, respectively, in~\cite{mishra2012} and~\cite{huang2018blind}. % 
The stepsize $\ssz_t$ in each iteration
is obtained following a backtracking line search procedure with respect to the
line search condition~\eqref{eq-throptmk:ls-armijo}. In this backtracking procedure,
the initial trial stepsize $\ssz_t^0$ (line~\ref{algline:qprecon-rgd-stepsize}) is
important to the time efficiency of algorithm. We consider the following methods for
setting the initial trial stepsize, including notably % 
(i) exact line minimization (\eg,~\cite{Vandereycken2013b}), 
and (ii) Riemannian Barzilai--Borwein (RBB) stepsize rules~\cite{iannazzo2018riemannian}, which 
have been shown to be very efficient.
The following two RBB rules are considered,  
\begin{equation}\label{eq-throptmk:stepsize_bb}
    {\ssz_t^{\mathrm{BB}1}} := \frac{\tmet[x][t](\bar{z}_{t-1},\bar{z}_{t-1})}{|
    \tmet[x][t](\bar{z}_{t-1},\bar{y}_{t-1})|}, \quad\quad
{\ssz_t^{\mathrm{BB}2}}  :=\frac{|  \tmet[x][t](\bar{z}_{t-1},\bar{y}_{t-1}
)|}{\tmet[x][t](\bar{y}_{t-1},\bar{y}_{t-1})},
\end{equation}
where $\bar{z}_{t-1} = \bar{x}_t - \bar{x}_{t-1}$ and $\bar{y}_{t-1} = \trgrad -
\trgrad[x][t-1]$.

\begin{algorithm}[htbp]
\caption{Riemannian Gradient Descent (RGD) using the preconditioned metric}\label{algo-throptmk:generic-rgd}
\begin{algorithmic}[1]
\REQUIRE{Initial point $\bar{x}_0\in\mtotal$, parameters $\beta, \sigma\in (0,1)$, $\ssz_0>0$ and $\epsilon > 0$; $t=0$.}
\ENSURE{$\bar{x}_t\in\mtotal$.} 
\WHILE{$\|\trgrad[x][t]\|>\epsilon$}
\STATE{Set $\ttanv[\eta]_t =-\trgrad[x][t]$ using~\eqref{eq-throptmk:rgrad-man-precon}.} 
\STATE{Backtracking line search: % 
find the smallest integer $\ell\geq 0$ such that, for $\ssz_t:= \ssz_{0}\beta^{\ell}$, 
\begin{equation}\label{eq-throptmk:ls-armijo}
    \bar{f}(\bar{x}_t) - \bar{f}\left(\bar{x}_t + \ssz_t \ttanv[\eta]_t\right) \geq \sigma \ssz_t \bar{g}_{\bar{x}_{t}}(-\trgrad[x][t],\ttanv[\eta]_t).
\end{equation}
}\label{algline:qprecon-rgd-stepsize}%
\vspace*{-1.4em}
\STATE{Update: $\bar{x}_{t+1} = \bar{x}_t + \ssz_t \ttanv[\eta]_t$; $t\leftarrow
t+1$.}\label{algl-throptmk:rgd-update}
\ENDWHILE{}
\end{algorithmic}
\end{algorithm}

\paragraph{Riemannian conjugate gradient descent (RCG)}
\label{ssec:alg-rcg}
Based on the same computational elements as for Algorithm~\ref{algo-throptmk:generic-rgd}, we also consider a Riemannian conjugate gradient (Qprecon~RCG) algorithm on the total space $\mtotal$. With the same Riemannian gradient definition, the search direction $\ttanv[\eta]$ of the RCG algorithm at the $t$-th iteration is defined as 
$\ttanv[\eta]_{t} = -\trgrad[x][t] + \beta_t\ttanv[\eta]_{t-1}$, %
where $\beta_t$ is computed using a Riemannian version of one of the nonlinear conjugate gradient rules. 
In the numerical experiments, we choose the Riemannian version of the following modified Hestenes-Stiefel rule~\cite{Hestenes&Stiefel:1952} (HS+), 
$\beta = \max\big(0, \frac{\tmet[x][t](\ttanv_t-{\ttanv}_{t-1}, 
{\ttanv}_t)}{\tmet[x][t](\ttanv_t-\ttanv_{t-1},\ttanv[\eta]_{t-1})}\big)$, 
where $\ttanv_t:=\trgrad[x][t]$ for $t\geq 0$.

\subsection{Geometry of \texorpdfstring{$\mrkk$}{Mk} as the quotient of \texorpdfstring{$\mtotal$}{the product space}}\label{ssec:geom-manf-as-quosp}

As briefly mentioned in Section~\ref{sec:intro}, the projection $\pi:\mtotal\to\mrkk$: $(G,H) \mapsto G\trs[H]$ induces the equivalence relation $\sim$ on $\mtotal$. 
The structure of the fibers of $\pi$ can be characterized by the linear group $\glin[\rkval]$: %
in fact, the operations $(G,H)\mapsto (GF^{-1}, H\trs[F])$ for $F\in\glin[\rkval]$ correspond to all possible transformations that leave the matrix product $X=G\trs[H]$ unchanged. This means that %
$X=\pi((G,H))$ is on an one-to-one correspondence with %
$\{ (GF^{-1},H\trs[F]): F\in\glin[\rkval] \}$, 
which defines the equivalence class $\{(G',H')\in\mtotal: G'\trs[H'] = G\trs[H]\}$. % 
Hence the identification $\mrkk\simeq\mtotal/\glin[\rkval]$. Moreover, the quotient space $\mrkk$ is a quotient submanifold of $\mtotal$~\cite[\S3.4]{AbsMahSep2008}. %
The product space $\mtotal$ is referred to as the \textit{total space}. 

The structure of the tangent space to $\mrkk$ can be deduced as follows~\cite[\S3]{AbsMahSep2008}. 
Given a matrix $X\in\mrkk$ and a tangent vector $\xi\in\tansp[X]\mrkk$, the mapping $\pi:\mtotal\to\mrkk$ induces infinitely many representations of $\xi$: 
for $\bar{x}\in\pi^{-1}(X)$, any element $\ttanv\in\ttansp$ that satisfies $\dop\pi(\bar{x})[\ttanv] = \xi$ can be considered as a representation of $\xi$. Indeed, for any smooth function $f:\mrkk\to\reals$, the function $\bar{f}:=f\circ\pi:\mtotal\to\reals$, one has the following identification,
\[
\dop\bar{f}(\bar{x})[\ttanv] = \dop f(\pi(\bar{x}))[\dop\pi(\bar{x})[\ttanv]] =
\dop f(X)[\xi].
\]
Since $\pi:\mtotal\to\mrkk$ is surjective, the kernel of $\dop\pi(\bar{x}):\tansp[\bar{x}]\mtotal\to\tansp[X]\mrkk$ is non-trivial, the matrix representation of $\xi$ in $\ttansp$ is not unique. 
Nevertheless, one can find a unique representation of $\xi$ in a subspace of $\ttansp$. This is realized by decomposing the tangent
space $\ttansp\simeq\prodsp$ into two \revv{complementary} subspaces, \ie, $\ttansp := \tversp
\oplus\thorsp$, where $\tversp:= \tansp[\bar{x}](\pi^{-1}(X))$, called the {\it vertical space}, is the tangent space at $\bar{x}$ to the equivalence class $[\bar{x}]$, 
and $\thorsp$, called the {\it horizontal space}, is the complementary of $\tversp$ in $\ttansp$. 
One can see that a tangent vector $\ttanv\in\tversp$ satisfies
$\dop\pi(\bar{x})[\ttanv] =0$. % 

Consequently, for any $X\in\mrkk$ and $\xi\in\tansp$, there is a unique representation
$\ttanv\in\thorsp\subset\ttansp$ of $\xi$ such that
$\dop\pi(\bar{x})[\ttanv] = \xi$. 
The tangent vector $\ttanv\in\thorsp$ is called the horizontal lift of $\xi$. 

Given the horizontal lifts as the matrix representation of tangent vectors to the quotient manifold $\mrkk$, % 
any metric $\bar{g}$ on the total space that satisfies following invariance property induces a metric on $\mrkk$. 
\begin{definition}[{\cite[\S3]{absil2014two}}]\label{def-throptmk:invariant-metric}
For $X\in\mrkk$ and 
$\bar{x}\in\pi^{-1}(X)$, let $\ttanv, \ttanv[\eta]\in\ttansp$ denote the horizontal lifts of $\xi$ and $\eta$ respectively, 
a metric $\bar{g}$ in the total space is said to be \emph{invariant} along $\pi^{-1}(X)$ if 
$\tmet[x](\ttanv,\ttanv[\eta]) = \bar{g}_{\bar{y}}(\ttanv_{\bar{y}},\ttanv[\eta]_{\bar{y}})$ 
for any point $\bar{y}\sim\bar{x}$ in $\pi^{-1}(X)$, where $\ttanv_{\bar{y}}, \ttanv[\eta]_{\bar{y}}$ denote the horizontal lifts of
$\xi$ and $\eta$ at $\bar{y}$ respectively. 
\end{definition}
Indeed, given a metric $\bar{g}$ that satisfies the invariance property in Definition~\ref{def-throptmk:invariant-metric}, 
the inner product  
$g_{X}:\tansp\mrkk\times\tansp\mrkk\to\reals$ such that $g_{X}(\xi,\eta)=\bar{g}_{\bar{x}}(\ttanv_{\bar{x}},\ttanv[\eta]_{\bar{x}})$ is a metric on the quotient manifold $\mrkk$. 
The following proposition shows that the metric~\eqref{eq-throptmk:def-metric-precon} induces a metric on $\mrkk$. 
\begin{proposition}\label{prop:invariance-metric-precon}
For any matrix $X\in\mrkk$, the preconditioned metric~\eqref{eq-throptmk:def-metric-precon} satisfies the invariance property as in Definition~\ref{def-throptmk:invariant-metric}, that is, for \revv{two horizontal lifts} $\bar{x}\in\mtotal$ and $\ttanv,\ttanv[\eta]\in\ttansp$, 
$\bar{g}_{\bar{x}}(\ttanv,\ttanv[\eta]) = \bar{g}_{\bar{x}'}(\ttanv[\xi'],\ttanv[\eta'])$, 
for any $\bar{x}'\sim\bar{x}$, where $\ttanv[\xi'],\ttanv[\eta']\in\ttansp[\bar{x}']$ are the horizontal lifts 
at $\bar{x}'$ 
such that $\dop\pi(\bar{x'})[\ttanv[\xi']]=\dop\pi(\bar{x})[\ttanv]$ and $\dop\pi(\bar{x}')[\ttanv[\eta']]=\dop\pi(\bar{x})[\ttanv[\eta]]$.
\end{proposition}

\section{Main results about RGD on \texorpdfstring{$\mrkk$}{Mk}}\label{sec-ropt-mk:lrqp}

In this section, we investigate how the matrix factorization-based RGD algorithm
(\cref{algo-throptmk:generic-rgd})---designed for solving the reformulated problem~\eqref{prog:main-generic}---behaves in relation with the landscape
of the original problem 
\begin{equation}
    \label{prog:main-mk}
    \min_{X\in\mrkk} ~ f(X).
\end{equation}
The matrix factorization approach~\eqref{prog:main-generic} converts the rank constraint of~\eqref{prog:main-mk} on $\vecsp$ into an unconstrained problem on $\prodspstar$, %
which allows for efficient algorithms with a significant reduction in both memory (from $O(mn)$ to $O((m+n)\rkval)$) and computation time costs. % 

However, some difficulties persist with the reformulation~\eqref{prog:main-generic}. 
Note that the {\it lifted} objective function $\bar{f}$ is invariant along the
equivalence classes in $\mtotal$, \ie, $f(\bar{x})=f(\bar{x}')$, for any
$\bar{x}'\sim\bar{x}$. % 
This poses issues to algorithms for~\eqref{prog:main-generic}. For example, due to the invariance of $\bar{f}$ in equivalence classes, any (isolated) local minimum of $f$ on $\mrkk$ corresponds to a whole equivalence class in $\mtotal$, which contains an infinite number of nondegenerate minima. This makes the landscape of $\bar{f}$ on $\mtotal$ different and actually more complicated than~\eqref{prog:main-mk}, even under the smallest dimensions (see~\Cref{fig:lands-mf}). %
Interestingly, we will deduce concrete results in Section~\ref{ssec:rgrad-mk}--\ref{ssec:algseq-mtotal2mquo} to show that the RGD algorithm
(\cref{algo-throptmk:generic-rgd}) bypasses the above difficulties, using the
invariance property (\cref{prop:invariance-metric-precon}) of the preconditioned metric.

\subsection{Riemannian gradient on the quotient manifold}\label{ssec:rgrad-mk}

Under the preconditioned metric $\bar{g}$~\eqref{eq-throptmk:def-metric-precon}, the gradient vector field of $\bar{f}$ induces a Riemannian gradient vector field of $f$ in the tangent bundle of $\mrkk$. 
Indeed, from \cref{prop:invariance-metric-precon}, the metric $\bar{g}$~\eqref{eq-throptmk:def-metric-precon} is invariant
along the equivalence classes and therefore induces a metric $g$ in $\mrkk$
such that, for any $X\in\mrkk$ and $\xi,\eta\in\tansp[X]\mrkk$, 
\begin{equation}\label{eq:id-met-total-quot}
g_{X}(\eta,\xi) = g_{\bar{x}}(\ttanv[\eta],\ttanv),
\end{equation}
where $\bar{x}$ is an element in $\pi^{-1}(X)$ and $\ttanv[\eta]$ and $\ttanv$
are the horizontal lifts (at $\bar{x}$) of $\eta$ and $\xi$ respectively. 
\bleu{It follows that, 
for any $\bar{x}\in\pi^{-1}(X)$, $\ttanv[\eta]:=\trgrad$ and 
any $\ttanv\in\ttansp$: 
\begin{align*}
g_{\pi(\bar{x})}(\xi, \rgrad[\pi(\bar{x})]) = 
\bar{g}_{\bar{x}}(\ttanv[\eta],\ttanv) = \dop\bar{f}(\bar{x})[\ttanv] 
= \dop(f\circ\pi)(\bar{x})[\ttanv] = \dop f(\pi(\bar{x}))[\dop\pi(\ttanv)]. 
\end{align*}
Hence the horizontal component of
$\bar{\eta}=\trgrad$ is the horizontal lift of $\rgrad$ at $\bar{x}\in\pi^{-1}(X)$. 
Note that $\trgrad$ belongs to the horizontal space $\thorsp$: since $\bar{f}$ is invariant along the equivalence classes, it is constant on $\pi^{-1}(X)$, which entails that for any $\ttanv\in\tversp=\tansp[\bar{x}](\pi^{-1}(X))$, $\dop\bar{f}(\bar{x})[\ttanv]=\bar{g}_{\bar{x}}(\ttanv,\trgrad)= 0$. 
Therefore, $\trgrad$ is the horizontal lift of $\rgrad$ at $\bar{x}$. 

As a consequence, all horizontal lifts $\trgrad$ for different
$\bar{x}\in\pi^{-1}(X)$ are equivalent. 
This means the induced Riemannian gradient $\rgrad$ is invariant to the location of $\bar{x}$ in the equivalence
class $\pi^{-1}(X)$; which can be reflected by an explicit form of $\rgrad$
independent of $\bar{x}$ as follows. 
} 

\begin{proposition}\label{prop:rgrad-total2quot}
Given a matrix $X\in\mrkk$, let $X=U\Sigma \trs[V]$ denote its $\svd$. % 
Let $g$ be the metric on $\mrkk$ induced by the preconditioned metric~\eqref{eq-throptmk:def-metric-precon}, the Riemannian gradient of $f$ on the quotient manifold $(\mrkk,g)$ satisfies 
\begin{equation}\label{eq:rgrad-as-pegrad}
 \rgrad = \opvx(\nabla f(X)),%
\end{equation}
where $\nabla f(X)$ is the Euclidean gradient of $f$ in $\reals^{m\times n}$ and $\opvx:\vecsp\to\tansp[X]\mrkk$ is a linear operator such that, for any $Z\in\vecsp$, 
\begin{equation}\label{eq:ov-egrad2rgrad}
 \opvx(Z) = \pu Z + Z\pu[V], %
\end{equation} 
where $\pu:=U\trs[U]$ and $\pu[V]:=V\trs[V]$ are the orthogonal projections onto the column and row subspaces of $X$ respectively. 
\end{proposition}
\begin{remark}\normalfont 
    \Cref{prop:rgrad-total2quot} is a result
    in the thesis~\cite[\S5.2.3]{Dong2021thesis}. In a more recent work~\cite[Remark 9]{luo2021geometric}, the same
    result is given as an example of the geometric
    connections of embedded and quotient geometries in Riemannian fixed-rank matrix optimization. 
\end{remark}
\begin{remark}\label{rmk:ov-egrad2rgrad}
    \normalfont
In the notation of~\eqref{eq:ov-egrad2rgrad}, $\pu$ and $\pu[V]$ are the matrices of the orthogonal projections. 
We also denote by abuse the action of these projection operators as $\pu(Z)=\pu Z$ and $\pu[V](Z)=Z\pu[V]$ respectively. 
The action of $\opvx$~\eqref{eq:ov-egrad2rgrad} can be rewritten as
$\opvx(Z) = \pu Z {(I-\pu[V])} + {(I-\pu)} Z \pu[V] + 2 \pu Z \pu[V]$. 
This implies that the operator $\opvx$ is related to the orthogonal projection 
$\proj_{\tansp[X]\mrkk[k]}$ 
(see~\cite[(2.5)]{Vandereycken2013b}) by 
\begin{align}\label{eq:rmk-rgrad-vs-proj}
\opvx(Z)=\proj_{\tansp[X]\mrkk[k]}(Z) + \pu Z\pu[V],\text{~for all~} Z\in\vecsp.
\end{align}
\end{remark}
The next lemma relates the operator $\opvx$ to  
the orthogonal projection operator $\proj_{\tansp[X]\mrkk}$. 

\begin{lemma}\label{lemm:opvx-props}
Let $X$ be a matrix in $\mrkk$ and let $X:=U\Sigma\trs[V]$ denotes its rank-$\rkval$ $\svd$. The linear operator $\opvx:\vecsp\to\tansp[X]\mrkk$ defined in~\eqref{eq:ov-egrad2rgrad} satisfies the following properties. 

(i): $\opvx$ is symmetric and positive semidefinite. In particular, 
$\opvx(Z)=2Z$ if $Z\in\tansp[X]\mrkk$, and $\opvx(Z)=0$ if $Z\in(\tansp[X]\mrkk)^{\perp}$. 

(ii): Let $X,Y$ be two matrices in $\mrkk$, and 
let $(\opvx-2\id)(Z):=\opvx(Z)-2Z$ denote the evaluation of the operator $(\opvx-2\id)$ at $Z\in\vecsp$. 
It holds that  
	\begin{align}
		(\opvx-2\id)(X-Y) = 2(\id-\proj_{\tansp[X]\mrkk})(Y),\label{eq:opvx-iddiff}
	\end{align}
where $\proj_{\tansp[X]\mrkk}$ is the orthogonal projector onto the tangent
space to $\mrkk$ at $X$. 
\end{lemma}

% *****************************************************************************

\subsection{Induced sequence in the quotient manifold}\label{ssec:algseq-mtotal2mquo}

Now, we are ready to deduce the image of the sequence $\algseq\subset\mtotal$ by RGD (\Cref{algo-throptmk:generic-rgd}) under the projection $\pi$. We refer to the image $\{\pi(\bar{x}_t)\}_{t\geq0}$ as the {\it induced}\/ sequence, which is denoted by $\indseq\subset\mrkk$ with $X_{t}=\pi(\bar{x}_{t})=G_{\bar{x}_t}\trs[{H_{\bar{x}_t}}]$.

\begin{lemma}\label{lemm:rgd-mtotal2manf}
Let $(\bar{x}_{t})_{t\geq 0}\in\mtotal$ denote the sequence generated by Algorithm~\ref{algo-throptmk:generic-rgd} and let $\indseq$ be the induced sequence in $\mrkk$. 
Then the iterates satisfy 
\begin{equation}\label{eq:lemm-rgd-mtotal2manf}
  % X_{t+1} = X_t - \ssz_t \vfx[X_t] + \ssz_t^{2} \nabla f(X_t){X_t}^{\dagger}\nabla f(X_t),
  X_{t+1} = X_t - \ssz_t \rgrad[X_t] + \ssz_t^{2} \nabla f(X_t){X_t}^{\dagger}\nabla f(X_t),
\end{equation}
where $X_t^{\dagger}\in\reals^{n\times m}$ denotes the Moore--Penrose pseudoinverse of $X_t$. 
\end{lemma}

Note that the characterization in \Cref{prop:rgrad-total2quot} is {\it not}\/ a computational component of the RGD algorithm (\cref{algo-throptmk:generic-rgd}) but is an update rule on $\mrkk$ induced by the algorithm. 
The quotient manifold-based algorithms in the framework of Section~\ref{ssec:alg-details} (such as Algorithm~\ref{algo-throptmk:generic-rgd}) produce iterates in the total space $\mtotal$. 

The invariance property in \Cref{prop:rgrad-total2quot} and \Cref{lemm:rgd-mtotal2manf} enables us to qualify these algorithms as Riemannian descent algorithms on the quotient space $\mrkk$. 
\bleu{The explicit form in \Cref{lemm:rgd-mtotal2manf} % about the induced sequence of this algorithm 
allows for a local convergence analysis on $\mrkk$ that is substantially
different and simpler than those based on the specific matrix factorization of
$X\in\matsp$, as we present next. %
}

\section{Convergence behavior in low-rank matrix recovery}\label{ssec-ropt-mk:lrqp-intro}

In this section, we consider the problem~\eqref{prog:main-mk} with a quadratic objective function $f$ as follows: 
\begin{equation}
    \label{eq:def-f-lrqp}
    f(X) := \frac{1}{2}\braket{X,\opa(X)} - \braket{\bstar, X},
\end{equation}
where $\opa:\vecsp\to\vecsp$ is a linear operator that is symmetric and positive semidefinite, $\bstar = \opa(\mstar)$
is a data matrix containing observations of an unknown matrix $\mstar\in\vecsp$, and 
$\braket{\cdot,\cdot}$ is the Frobenius inner product of two matrices. 
 
The function $f$~\eqref{eq:def-f-lrqp} generalizes the objective function of many matrix recovery 
problems~\cite{donoho2006compressed,candes2009exact,Fazel2008,Sun2016a,CaiEtAl2017Spectral,huang2018blind}.
For compressed sensing (Example~\ref{exa:lrqp-ms}), the objective function 
    \begin{equation}\label{eq:def-f-cs}
        f(X) = \frac{1}{2}(\|\Phi(X)-b^{\star}\|_{2}^2-\|b^{\star}\|_{2}^{2}) 
    \end{equation}
can be written in the form of~\eqref{eq:def-f-lrqp}, with $\opa =
\duop[\Phi]\Phi$ and $\bstar=\duop[\Phi]b^{\star}$. 
For matrix completion (Example~\ref{exa:opt-mk-via-mtotal}), the recovery
of the hidden matrix $\mstar$ can be realized by the minimizing the following quadratic function 
   \begin{equation}\label{eq:def-fit-mc}
   f(X) = \frac{1}{2p}(\fro{\pom(X-\mstar)}^2 - \fro{\pom(\mstar)}^2),
   \end{equation}
where $\Omega$ is the index set of the known entries, $p=|\Omega|/mn$, and
$\po:\vecsp\to\vecsp$ is the projection operator that retains only the
entries on $\Omega$ and projects all other entries to zero. In this case, $\opa
=\frac{1}{p}\pom$ and $\bstar=\frac{1}{p}\pom(\mstar)$. 

\subsection{Critical points}

A matrix $\xstar\in\mleqk$ is a critical point for the minimization of $f$~\eqref{eq:def-f-lrqp} on $\mleqk$ if the Euclidean gradient $\nabla f(\xstar)$ belongs to the polar cone of the tangent cone of $\mleqk$ at $\xstar$. 
As characterized in~\cite{schneider2015convergence} and \cite[Proposition~2.4]{10.1093/imanum/drz061}, such a critical point $\xstar$ can be a critical point of $f$ on $\mrkk$, or a solution to the (global) optimality condition $\opa(X)-\bstar=0$. 

In general, the operator $\opa$ of~\eqref{eq:def-f-lrqp} admits a nontrivial kernel and %
there are thus more than one point in $\mleqk$ that satisfy either of the above two characterizations. 
In particular, the hidden matrix $M^\star$ is a critical point and also the global optimum of~\eqref{prog:main-mk}, as long as the rank of $\mstar$ is equal to $\rkval$ in~\eqref{prog:main-mk}. This scenario is often referred to as the {\it noiseless case}\/ in the low-rank matrix recovery problems.

Recent advances in low-rank matrix recovery have results about guarantees of recovering $\mstar$ via~\eqref{prog:main-mk}. It is shown in \cite[Theorem~3.5 and corollaries]{10.1093/imanum/drz061}, under the RPD property with a bounded RPD constant $\brip_{3k}\leq C_\mu <1$ where $C_\mu$ is a constant depending on the Riemannian Hessian of $f$~\eqref{eq:def-f-lrqp} at $\mstar$, that: 
(i) the optimality equation $\opa(X)=\bstar$ admits $\mstar$ as the unique
solution % 
(ii) $\mstar$ is the unique global minimum of $f$ and %
(iii) all other critical points $X'\neq \mstar$ of $f$ are strict saddle points. % 
This implies, in particular, that 
one can use Riemannian descent algorithms on $\mrkk$ to find $\mstar$, provided that the RPD constant $\brip$ satisfies the given upper-bound condition. %
 
Instead of studying the uniqueness of $\mstar$ (as a local minimum of $f$ on $\mrkk$), we focus on the properties of $f$ 
near $\mstar$ without requiring $\mstar$ to be the unique local minimum. % 
This motivates us to make a refined local convergence analysis with respect to the hidden matrix $\mstar$, as we will present in Section~\ref{ssec-ropt-mk:convana}. 
The RPD property (\Cref{def:lrqp-rip}) with respect to $\mstar\in\mrkk$ is used but without a specific bound on the RPD constant $0\leq\brip<1$.  

\renewcommand{\xstar}{M^{\star}}

\begin{proposition}\label{prop-rpd2asp}
% Assumption~\ref{asp:rpd-xstar} holds for any $\mstar\in\mrkk$ if 
Suppose $\opa$ satisfies the $(\brip,2\rkval)$-$\rpd$ property. Then 
for any $\mstar\in\mrkk$ 
\begin{equation}\label{eq:ineq-asp-rpd}
   (1-\brip)\frobb[X-\xstar]\leq
   \braket{X-\xstar,\opa(X-\xstar)}\leq(1+\brip)\frobb[X-\xstar] 
\end{equation}
   holds for $X\in\mrkk$. 
\end{proposition}
\begin{proof}
    This is because $(X-\mstar)\in\mleqk[2\rkval]$ for any $X$ and $\mstar$ in $\mrkk$. 
\end{proof}

\subsection{Convergence analysis}\label{ssec-ropt-mk:convana}

Through \Cref{lemm:rgd-mtotal2manf}, the study of convergence properties of the RGD algorithm (\Cref{algo-throptmk:generic-rgd}) on the product space $\mtotal$ can be conducted instead on $\mrkk$ directly. 
Here, $\mrkk$ is considered as an embedded manifold of $\reals^{m\times n}$ now
that all objects in the quotient manifold representation have explicit
expressions through \Cref{ssec:rgrad-mk}--\Cref{ssec:algseq-mtotal2mquo}.
Proofs to the results in this section are given in \Cref{sec-app:prf-sec4}.

In the noiseless case, \ie, the hidden matrix $\mstar$ has a low rank $\rkval$, we investigate the Riemannian Hessian of $f$~\eqref{eq:def-f-lrqp} at $\xstar$. 

% ****************
\begin{lemma}\label{lemm:f-hess-pd}
Let $f$ be a quadratic function~\eqref{eq:def-f-lrqp} such that the hidden matrix $\xstar$ has a low rank $\rkval$ and 
$\opa$ satisfies the $(\brip, 2k)$-RPD propety. 
Then the Riemannian Hessian of $f$ at $\xstar$ is positive definite: 
\begin{equation}\label{eq:ineq-hess-lrqp}
    \lambda_{\min}(f):=\min\limits_{W\in\tansp[\xstar]\mrkk}\left(\frac{\braket{\hess
    f(\xstar)[W],W}}{\frobb[W]}\right)\geq 1-\brip>0.
\end{equation}
\end{lemma}
 
% The proof of this lemma is given in \Cref{ssec-ropt-mk:prfs-mainres}. 

% ***************
Consequently, the following immediate result about \Cref{algo-throptmk:generic-rgd} can be deduced, provided that the algorithm converges.  
\begin{theorem} 
    \label{thm:main-generic}
Let $f$ be a quadratic function~\eqref{eq:def-f-lrqp} such that the hidden matrix $\xstar$ has a low rank $\rkval$ and 
$\opa$ satisfies % 
the $(\brip,2k)$-RPD propety. 
Then, if the sequence $\indseq$ induced by \Cref{algo-throptmk:generic-rgd} converges to $\xstar$, the local convergence rate of $\indseq$ is linear. 
\end{theorem}

\subsubsection{Local convergence analysis} 

The convergence result of \Cref{thm:main-generic} requires the additional assumption that the algorithm
converges to $\mstar$, {\it a priori}, because the sole %
RPD property (with an unspecified parameter $\brip$) 
does not rule out the existence of other critical
points of $f$ on $\mrkk$ different than $\mstar$. %
Also, it is possible that the sequence does not admit an accumulation point due to the openness of $\mrkk$. The similar assumption is also seen in~\cite[2.4.4]{schneider2015convergence}. 
The result above % 
covers a more general class of functions than~\eqref{eq:def-f-lrqp} and thus
does not take the local properties of $f$~\eqref{eq:def-f-lrqp} into
consideration. 
\redl{To go further in the case of quadratic functions, we investigate local
convergence of RGD on $\mrkk$ in a certain neighborhood of $\mstar$:} 
\begin{equation}\label{def:nm-bar}
  \nmb=\{X\in\mrkk: \fro{X-\xstar}\leq \dnm\}. 
\end{equation}

\begin{lemma}% 
\label{lemm:f-lip-muu}
Let $f$ be a quadratic function~\eqref{eq:def-f-lrqp} such that the hidden
matrix $\xstar$ has a low rank $\rkval$ and $\opa$ satisfies the $(\brip,
2k)$-RPD propety. Then: 
\begin{itemize}
\item[(i)] %
    The quadratic function $f$~\eqref{eq:def-f-lrqp} has Lipschitz continuous gradient on $\mleqk$: there exist $\lip>0$ and $\rlip>0$ depending on $(\brip, k)$ such that  
 $\frob[\nabla f(X) - \nabla f(Y)] \leq \lip \frob[X - Y]$ for all $X,Y\in\mleqk$, and  
\begin{align}% 
f(Y) \leq f(X) + \braket{\grad f(X),  Y-X} +  \frac{\rlip}{2} \frob[Y - X]^2, 
\quad \forall ~X,Y \in \mleqk.
\label{eqs:f-lipschitz-rsc-a}
\end{align}

\item[(ii)] %
    \redl{There exists a constant $\muu > 0$ depending only on $f$ such that   
    \begin{align}
        f(X) - f(\mstar) \leq \braket{\grad f(X), X-\mstar} - \frac{\muu}{2} \fro{X-\mstar}^2
        \label{eq:tar-mu-pre}
    \end{align}
    for all $X\in \mrkk$ with $\fro{X-\mstar}\leq \frac{1}{2}\cbd$.} 
\checkk{~\\ (To verify) This is {\it not} applicable to any $(X,Y)$: Try adapt
inequality (4.7) within this ``$B_{\delta}$'' using $L_f$ instead of $L_g$: 
$$f(\mstar) - f(X) \leq \braket{\grad f(X), \mstar-X} + \frac{\lip}{2} \fro{X-\mstar}^2.$$}
\end{itemize}
\end{lemma}
\redl{Inequality~\eqref{eq:tar-mu-pre} can be seen as a slightly weaker form of
local strongly convexity; in other words, we have 
$ f(Y) - f(X) \geq \braket{\grad f(X), Y-X} + \frac{\muu}{2} \fro{Y-X}^2$ 
for $Y := \mstar$ and for all %
$X \in \mrkk$ satisfying the above closeness condition. 
This property is related to but differs from {\it geodesically strongly convexity} in the sense
that we choose not to use the exponential mapping $\text{Exp}_X^{-1}(Y)$ on the
manifold $\mrkk$, but just use $(Y-X)$ from the perspective of $\mrkk$ being an embedded manifold
of $\matsp$. 
}

In the following lemma, we provide a sufficient condition under which $\mstar$ is the unique critical point in $\nmb$. 
\bleu{The reasoning is that when $f$~\eqref{eq:def-f-lrqp} satisfies the RPD property, 
it % 
behaves similarly enough as the squared distance function $X\mapsto\frac{1}{2}(\frob[X-\bstar]^{2}-\fro{\bstar}^2)$, which is strongly convex in the ambient space $\vecsp$. 
We show that this property in $\vecsp$ is preserved locally on the manifold $\mrkk$. 
}

\begin{lemma}\label{lemm:f-rsc-riem}
Let $f$ be a quadratic function~\eqref{eq:def-f-lrqp} such that the matrix $\xstar$ has rank $\rkval$ and 
$\opa$ satisfies the $(\brip,2k)$-RPD propety. 
Then it holds that 
\begin{align}
\label{eq:rgrad-nmb}
\redl{ \braket{\rgrad[X], X - \xstar} > 0, } 
\end{align}
for all $X$ in the interior of $\nmb[\dnm_0]$~\eqref{def:nm-bar} with $\dnm_0 = \cbd$, where 
$\lip$ is the Lipschitz constant of $f$ in
\cref{lemm:f-lip-muu} and $\sigmin:=\sigma_{\rkval}(\mstar)>0$. 
Moreover, $\mstar$ is the unique critical point of $f$ in the interior of 
$\nmb[\dnm_0]$. %
\end{lemma}
The proofs of \Cref{lemm:f-lip-muu} (ii) and \Cref{lemm:f-rsc-riem} are mainly based on
\Cref{lemm:opvx-props}--\Cref{lemm:rgd-mtotal2manf} and \cite[Lemma 4.1]{wei2016guarantees} about the orthogonal projector~$\proj_{\tansp[X]\mrkk}$. 

\redl{
\begin{remark} 
    The value of the radius $\dnm_0=\cbd$ in \Cref{lemm:f-rsc-riem} is exactly twice the radius within which inequality \eqref{eq:tar-mu-pre} of \Cref{lemm:f-lip-muu} holds. Since the larger value ($\cbd$) will appear more often in the next few technical lemmas, we denoted it with the symbol $\dnm_0$. 
\end{remark} 
}

\paragraph{Sufficient descent conditions.}

Now we study conditions for RGD to ensure sufficient descents in a compact subset of the neighborhood $\nmb$~\eqref{def:nm-bar}, 
parametrized by a radius $ \delta >0 $ and a limit signular value $\bar{\sigma}>0$ 
as follows: 
\begin{align}
    \label{def:set-c}
C_{\delta,\bar{\sigma}}^\star:=\nmb\cap\{X: \sigma_k(X)\geq \bar{\sigma}\}. 
\end{align} 

For the expression of the RGD update \eqref{eq:lemm-rgd-mtotal2manf} (\Cref{lemm:rgd-mtotal2manf}), 
we use the following shortened notations whenever needed: 
$$X_{+}=X -\theta g_X + \theta^2 \Gamma_X.$$ 
Here, the subscript index $t$ and $t+1$ are omitted, 
$g_X:= \grad f(X)$ is short for the Riemannian gradient, and 
$\Gamma_X := \nabla f(X){X}^{\dagger}\nabla f(X)$ denotes the higher-than-first order term.

We apply the inequality of \Cref{lemm:f-lip-muu} to 
two consecutive RGD iterates $X:=X_t$ and $X_+:=X_{t+1}$ given by~\eqref{eq:lemm-rgd-mtotal2manf}, which entails 
$f(X_+) \leq  f(X) - \theta\fro{g_X}^2 + \theta^2 \braket{g_X, \Gamma_X}  +
\frac{\rlip\theta^2}{2} \fro{g_X - \theta \Gamma_X}^2$. 
In other words, the residuals $R_t:=f(X_t)-f(\mstar)$ satisfy 
\begin{align} 
    R_{t+1}-R_t = f(X_+)-f(X) \leq  D(\theta):= - \theta\fro{g_X}^2 + \theta^2
    \braket{g_X, \Gamma_X} + \frac{\rlip\theta^2}{2} \fro{g_X - \theta \Gamma_X}^2,
\label{eq:ineq48-1}
\end{align} 
where $D(\theta)$ can be rewritten as 
\begin{align*}
D(\theta)  
= - \frac{\theta}{2} \Big( (2-\rlip\theta)\fro{g_X}^2 -
2(1-\rlip\theta)\braket{g_X, \theta\Gamma_X} - \rlip\theta \fro{\theta\Gamma_X}^2 \Big). 
\end{align*}  

\redl{On the other hand, through \Cref{lemm:f-lip-muu} (ii), 
we have for all $X\in\mrkk$ such that $\fro{X-\mstar} \leq \frac{1}{2}\cbd$:}
    \begin{align*} 
        R_t = f(X)  - f(\mstar) \leq \braket{g_X, X-\mstar} - \frac{\muu}{2} \fro{X-\mstar}^2 
    \end{align*}    
The inner product on the right-hand side, $\braket{g_X, X-\mstar}$, satisfies the following trigonometric equation: 
    \begin{align}  
            \braket{g_X, X-\mstar}  = \frac{1}{2\theta} (\fro{X-\mstar}^2 - \fro{X_+ - \mstar}^2) + \frac{\theta}{2} (\fro{g_X -\theta \Gamma_X}^2 + 2\braket{\Gamma_X, X-\mstar}). \nonumber 
    \end{align}
Therefore, % 
\newcommand{\td}{\tilde{D}(\theta)}
\newcommand{\ccd}{C_{\delta}}
        \begin{align}\label{eq:ineq48-2} 
            R_t 
             \leq \frac{1-\theta\mu}{2\theta} \fro{X-\mstar}^2 - \frac{1}{2\theta} \fro{X_+-\mstar}^2 %
             +  \underbrace{\frac{\theta}{2}\big(\fro{g_X-\theta\Gamma_X}^2 + 2\braket{\Gamma_X, X-\mstar}}_{\td}\big), 
        \end{align}
        where the sum of the last two terms is denoted as $\tilde{D}(\theta)$. In the notation of both $D(\theta)$ and $\tilde{D}(\theta)$, `$X$' is omitted for brevity. 

\checkk{--(TODO: The following paragraph could be made CLEARER)--} 

We proceed by checking the following two descent conditions in terms of $D(\theta), \tilde{D}(\theta)$ in~\eqref{eq:ineq48-1}--\eqref{eq:ineq48-2}: 
\begin{align}
    \label{eq:cond-1}  % 
    & -D(\theta) > 0   \quad \text{(descent condition)}
    \\
    \label{eq:cond-rho} % 
    & -\rho D(\theta) - \tilde{D}(\theta) \geq 0 \quad \text{(`strong' descent condition)}
\end{align}
for $\rho > 1$ and $0<\theta < \frac{2}{\rlip}$. 
Notice that these two conditions, combined with~\eqref{eq:ineq48-1}--\eqref{eq:ineq48-2}, entail the following inequality: 
\begin{align} \label{ineq:tar}
    \rho R_{t+1} - (\rho -1) R_t \leq \frac{1-\theta \mu}{2\theta} \fro{X - \mstar}^2  - \frac{1}{2\theta} \fro{X_+ - \mstar}^2, 
\end{align}
which will be used to deduce the convergence rate of $\{X_t\}_{t}$. The roles of $\rho>1$ and \eqref{ineq:tar} are similar to those in the analysis of \cite{zhang2016first} for geodescically convex optimization, but the difference is that the manifold $\mrkk$ in this work is an open space and \eqref{ineq:tar} will be applied in a compact subset of $\mrkk$. The next couple of lemmas are about descent properties of RGD in such a compact subset.  

\bleu{The following lemmas are used for checking the two decent conditions \eqref{eq:cond-1}--\eqref{eq:cond-rho}. 
} 
Based on % 
\Cref{lemm:f-rsc-riem}, 
the next proposition gives properties of $\grad f(X)$ and $\Gamma_X$ in $C_{\delta,\bar{\sigma}}^\star$~\eqref{def:set-c}.

\begin{proposition} % 
    \label{lemm:4.8x}  
For $\dnm_0 = \cbd$ and % 
$\bar{\sigma}>0$, there exist $0<\cca \leq 1$ and $0 < \tcca \leq 1$ such that 
for any $ X \in C_{\dnm_0,\bar{\sigma}}^\star$~\eqref{def:set-c}: 
$\fro{\grad f(X)} \geq 2 {\cca}\fro{\nabla f(X)}$, and 
when $f$ satisfies the $(\brip,2k)$-RPD property, 
    \begin{align}\label{eq:lemm-4.8x}
            & \fro{\grad f(X)} \geq 2 {\tcca}\fro{X-\mstar}\\ 
            & \fro{ \Gamma_X } \leq \frac{\lip}{2\cca\bar{\sigma}} \fro{\grad f(X) } \fro{X-\mstar}. 
         \label{eq:lemm-4.8yy}
    \end{align}
\end{proposition}
\begin{remark} 
    The inequality~\eqref{eq:lemm-4.8x} in the lemma above can be seen as the Lojaciewicz inequality (with exponent $\frac{1}{2}$) in a restricted neighborhood of $\mstar$.   
\end{remark}

\begin{lemma}%
\label{lemm:4.8z}
For $\dnm_0 = \cbd$ and $\bar{\sigma}>0$, 
there exists %
$\rho \asymp (1+\rlip)$ 
such that for any $X\in C_{\dnm_0,\bar{\sigma}}^\star$~\eqref{def:set-c} 
different than $\mstar$, 
the RGD update~\eqref{eq:lemm-rgd-mtotal2manf} from $X$ satisfies 
the descent conditions~\eqref{eq:cond-1}--\eqref{eq:cond-rho}, provided that 
the stepsize $0<\theta\leq (1-\frac{1}{\rho}) \frac{1}{\rlip}$ and that 
\begin{align} 
    \fro{ \Gamma_X} \leq  C_{\rho} \rlip \fro{\grad f(X)} \label{eq:p-2cond} 
\end{align}
where $C_{\rho} = \min(1,\frac{\rho}{\sqrt{(\rho-1)(\rho^2-1)}})$. 
\end{lemma} 
\begin{coro}
    \label{coro:4.8z}
    \checkk{(Note: retain the value of $\dnm_0$)~}For $\dnm_0 = \cbd$ and $0<\bar{\sigma}\lesssim \sigmin$, 
there exist $\rho \asymp (1+\rlip)$, $\ccb\asymp 1$ \redl{and $\bar{\delta}_{0} =\min(\frac{1}{2}\dnm_0, \ccb\dnm_0)$} such that 
for any $X\in C_{\bar{\delta}_0,\bar{\sigma}}^\star$~\eqref{def:set-c} different than $\mstar$, 
the RGD update~\eqref{eq:lemm-rgd-mtotal2manf} from $X$ 
satisfies the descent conditions~\eqref{eq:cond-1}--\eqref{eq:cond-rho}, provided that 
the stepsize $0<\theta\leq (1-\frac{1}{\rho}) \frac{1}{\rlip}$.
\end{coro}
The proofs are given in \Cref{prf:lemm-4.8z}.
\begin{remark} 
    \normalfont 
    In \Cref{lemm:4.8z}--\Cref{coro:4.8z}, the parameter $\rho$  is defined as $\rho = 1 + \frac{\rlip}{2\tcca}$, and  
    the parameter $\ccb$ in \Cref{coro:4.8z} is $\ccb= \frac{\rlip\bar{\sigma}}{\sigmin} \frac{2\cca C_{\rho}}{1-\brip}$,  where $\rlip$ is given in \Cref{lemm:f-lip-muu} and $\cca, \tcca$ are given in \Cref{lemm:4.8x}. 
The value of $C_{\rho}$ evolves with $\rho$ continously as follows: 
    $C_{\rho} = 1$ if $1<\rho \leq \rho^*\approx 2.24$ and $C_{\rho} = \frac{\rho}{\sqrt{(\rho-1)(\rho^2-1)}}  < 1$ otherwise ($\rho > \rho^*$). 
When $\rlip\approx \sqrt{1+\brip} < \sqrt{2}$ and $\tcca \approx 0.5$, for example, we have $\rho \approx 2.41$ and $C_{\rho} \approx 0.93$. The value of $\ccb$ under \Cref{coro:4.8z} varies around $1$ since $\frac{2\cca C_{\rho} }{1-\brip} \asymp 1$  and $\frac{\rlip \bar{\sigma}}{\sigmin} \asymp 1$. 
\end{remark}

\paragraph{Local linear convergence.}  
Given the lemmas above, we have the following theorem. %

\begin{theorem}\label{lemm:f-rsc-fort}
Let $f$ be a quadratic function~\eqref{eq:def-f-lrqp} such that the hidden
matrix $\xstar$ has a low rank $\rkval$ and that 
$\opa$ satisfies the $(\brip,2k)$-RPD propety. 
Then, there exist $\rho \asymp 1+\rlip$ and $\ccc\asymp 1$ such that given an initial point $\bar{x}_0\in\mtotal$ \redl{satisfying $\fro{\pi(\bar{x}_0)-\mstar}\leq \min(\frac{1}{2}\dnm_0, \ccc\dnm_0)$,} where 
$\dnm_0 = \cbd$, 
and an initial stepsize $\theta_{0}=(1-\frac{1}{\rho})\frac{1}{\rlip}$,  
\Cref{algo-throptmk:generic-rgd} converges to $\mstar$ linearly, and the sequence $\{X_t\}_{t\geq 0}$ induced by \Cref{algo-throptmk:generic-rgd} satisfies 
$$
f(X_t) - f(\mstar) \leq (1-\epsilon)^{t} C_{f,k}\fro{X_0-\mstar}^2
\quad\text{for~} t>0. 
$$ 
Here $\epsilon = \min(\frac{1}{\rho}, \theta_0\muu) \in (0,1)$ and $C_{f,k} >0$ are constants depending only on $(\brip, \rlip,\muu)$. 
\end{theorem} 

\begin{remark}
    \normalfont
    The radius $\dnm_0=\cbd$ in the initial condition of \Cref{lemm:f-rsc-fort} and its lemmas is proven in the related work to be attainable by methods such as the spectral initialization (\eg, \cite{keshavan2010matrix}); see, \eg, \cite[Claim V.2]{Sun2016a} and \cite[Lemma 15]{tong2021accelerating}. 
\end{remark}

\paragraph{Discussion.}
\redl{
In \cite{tong2021accelerating}, the convergence properties of a gradient descent algorithm called ScaledGD, using the same preconditioned gradient \eqref{eq-throptmk:rgrad-man-precon}, is analyzed in the context of low-rank matrix estimation. 
The local convergence result of \Cref{algo-throptmk:generic-rgd} given in \Cref{lemm:f-rsc-fort} is similar to \cite[Theorem 11]{tong2021accelerating}. 
More precisely, the initial closeness condition ($0.1\sigma_k(\mstar) \frac{\sqrt{\mu}}{\sqrt{L}}$ in \cite{tong2021accelerating}), the linear rate ($(1-\theta_0\mu)$ in \cite{tong2021accelerating}) and the range of stepsize  in both results are the same in the order of magnitudes (taking their difference in the search space and distance function into account). 
In particular, the convergence rate of \Cref{algo-throptmk:generic-rgd} and ScaledGD \cite{tong2021accelerating} do not depend on the condition number $\kappa$ of~$\mstar$, while Euclidean gradient descent algorithms \cite{park2018finding,bhojanapalli2016dropping} have linear convergence rates that are slower by a factor of $\kappa$.

This work differs with \cite{tong2021accelerating} in that our analysis is based on connections between embedded and quotient geometries in the optimization of $f$ over $\mrkk$ (\Cref{sec-ropt-mk:lrqp}). 
The analysis of \cite{tong2021accelerating}, on the other hand, is conducted on the product space $\reals^{m\times r} \times \reals^{n\times r}$; a special distance that is invariant on equivalence classes (\ie, $\{(GQ,HQ^{-\mathrm{T}}): Q\in\glin[r]\}$) is used. 
For the difference in abstractions, our proof techniques can be applied (\eg, descent condition in \Cref{lemm:4.8z}) or extended (\eg, geometric properties in \Cref{lemm:4.8x}) to other Riemannian metrics on $\mrkk$ while the techniques in \cite{tong2021accelerating} are more tailored for the gradient 
\eqref{eq-throptmk:rgrad-man-precon}. 
}

\subsection{Results adapted for low-rank matrix completion}\label{ssec-ropt-mk:appl-mc}

The matrix completion problem (\Cref{exa:opt-mk-via-mtotal}) on $\mrkk$ with an objective function~\eqref{eq:def-fit-mc} is: 
\begin{align*}% 
\min\limits_{X\in\mrkk}
f(X)=\frac{1}{2p}(\fro{\po(X-\mstar)}^2-\fro{\po(\mstar)}^2).  
\end{align*}

It is well understood that the subsampling operator $\pom$ can be badly conditioned with unbalanced sampling pattern or on matrices that are 
highly {\it coherent}~\cite[\S1.1.1]{candes2009exact}. In such case, $\pom$ does not satisfy the $\rpd$ property. 
Nevertheless, it is also shown in~\cite{candes2009exact} that the
$\rpd$ property can be satisfied by $\pom$ 
under the following assumptions: \rev{(i) the distribution of the observed entries follow either an uniform or a Bernoulli model: \eg, $(i,j)\in\Omega$ with probability
$p=\frac{|\Omega|}{mn}$, where $p$ is a given sampling rate, and (ii)} the unknown
matrix $\mstar$ is $\mu$-incoherent, as defined below. 

\begin{definition}[Incoherence~\cite{candes2009exact}]
    \label{def:incoh}
$Z=U\Sigma\trs[V]\in\vecsp$ is \emph{$\mu$-incoherent} if 
$\|\rowof{U}{i}\|_2 \leq \sqrt{\frac{\mu\rkval}{m}}$ and $ \|\rowof{V}{j}\|_2 \leq \sqrt{\frac{\mu\rkval}{n}}$ 
for all $(i,j)\in\iset{m}\times\iset{n}$.
\end{definition}
The incoherence constant $\mu$ is a bound that measures the maximal row norms of $U$ and $V$. Since the $\svd$ factor matrix $U$ (respectively $V$) is orthonormal, such that $\fro{U}^{2}=\sum_{i=1}^{m}\|\rowof{U}{i}\|_{2}^2\equiv\rkval$, the incoherence constant $\mu$ is small if the variations of the entries in $U$ and $V$ are moderate; on the contrary, $\mu$ is large if $U$ and $V$ has columns with {\it spikes}. It can be shown that the smallest possible incoherence constant is $\mu=1$.

\paragraph{The RPD inequality} 
As mentioned in Section~\ref{ssec-ropt-mk:lrqp-intro}, the subsampling operator $\pom$ generally does not satisfy the RPD property (\cref{def:lrqp-rip}). However, under certain conditions
on the subsampling pattern, the inequality~\eqref{eq:ineq-asp-rpd} 
in \Cref{prop-rpd2asp} 
can be satisfied in some neighborhood of the hidden matrix $\mstar$, provided that $\mstar$ is sufficiently incoherent.
Consider the following set adapted from~\cite[\S3]{Sun2016a}: %\redl{[TMC'21]}: 
\begin{align}\label{def:nm-sl15}
    \mathcal{B}^{\star}_{\delta'_0}(C_B,\mu) = \nmb[\delta'_0]\cap\Big\{X\in\mleqk: \max(\sqrt{m}\rinfnorm[X],\sqrt{n}\rinfnorm[{\trs[X]}]) \leq C_B\sqrt{\mu\rkval} \Big\},
\end{align}
where $\nmb[\delta'_0]$ is defined by~\eqref{def:nm-bar} %
for $\delta'_0:= \frac{\sigmin}{6C_d\rkval^{1.5}\kappa}$, $\kappa
:=\frac{\sigmax}{\sigmin}$, $C_d>0$ are constants, and $C_B>0$ is a parameter to be specified.

The following proposition, adapted from a key result in~\cite{Sun2016a}, validates the inequality~\eqref{eq:ineq-asp-rpd} 
on a neighborhood of $\mstar$ in the form of~\eqref{def:nm-sl15}.

\newcommand{\mus}{\mu}
\newcommand{\cst}{C^{\star}}

\begin{proposition}\label{prop:mc-rip}
Assume that a rank-$\rkval$ matrix $\mstar\in\vecsp$ 
is $\mus$-incoherent (with $\mus >0$). Suppose the condition number of $\mstar$ is
$\kappa$ and $\alpha= m/n\geq 1$. Then there exist numerical constants 
$C_0>0$ and $\cst > 1$ such that: if the indices in $\Omega$ are uniformly generated with size
\begin{align}
    \label{eq:mc-samplesize}
|\Omega|\geq C_0\alpha\rkval\kappa^{2}\max(\mus\log(n),\sqrt{\alpha}\rkval^{6}\mus^{2}\kappa^{4}),
\end{align}
then with probability at least $1- 2n^{-4}$,  
\begin{equation}\label{eq:mc-rip}
\frac{1}{3}\fro{X-\mstar}^{2} \leq \frac{1}{p}\fro{\pom(X-\mstar)}^{2} \leq 2\frobb[X-\mstar],
\end{equation}
for any $X\in\nmg[{C_B,\mus}]$ defined in~\eqref{def:nm-sl15} 
for $C_B := \sigmax \sqrt{2\cst \rkval }$. 
\end{proposition}

\begin{remark}
\normalfont The inequality~\eqref{eq:mc-rip} of \cref{prop:mc-rip} is an instance of the RPD inequality~\eqref{eq:ineq-asp-rpd}, since 
$\frac{1}{p} \fro{\po(X-\mstar)}^2 
=\braket{X-\mstar,
\mc{A}(X-\mstar)}$ for $\mc{A}=\frac{1}{p}\po$. Therefore, \cref{prop:mc-rip}
provides a condition on the index set~$\Omega$ under which the RPD inequality~\eqref{eq:ineq-asp-rpd} holds, hence, provides a way to applying the main result
(\cref{lemm:f-rsc-riem}--\cref{lemm:f-rsc-fort}) to the matrix completion
problem. 
\end{remark}

\newcommand{\dproj}[1][B]{\mathcal{P}_{#1}}
To ensure that the iterates of RGD remain in the domain $\nmg[{\epsilon,\mu}]$,
we adapt the normalization operator in \cite{Sun2016a,tong2021accelerating} from the product space of factor matrices to an $m\times n$ matrix space (\eg, $\mrkk$): 
given $X\in \reals^{m\times n}$, $\dproj(X) = D^{(1)}_B X D^{(2)}_B$, 
where $D^{(1)}_B:=D^{(1)}$, $D^{(2)}_B:=D^{(2)}$ are diagonal square matrices with compatible sizes such that  
    \begin{align}
            D^{(1)}_{ii} = \min(1, \frac{B}{\sqrt{m}\| \rowof{X}{i}\|_2}) \mathrm{~ ~and~ ~}
            D^{(2)}_{jj} = \min(1, \frac{B}{\sqrt{n}\| \colof{X}{j}\|_2}) 
    \end{align}
for a parameter $B>0$. 
By abuse of notation, $\dproj$ acts on a pair of factor matrices $\bar{x}=(G_{\bar{x}},H_{\bar{x}})\in\mtotal$ by row-wise normalizations $\dproj(\bar{x}) = (D^{(1)}_B G_{\bar{x}}, D^{(2)}_B H_{\bar{x}})$.
Note that the operator $\dproj$ depends on the input ($\bar{x}$ and/or $X$), which is omitted in the notation.

We consider the following variation of RGD using the normalization operator $\dproj$: in line 4 of \Cref{algo-throptmk:generic-rgd} (given $\bar{x}_t \in\mtotal$, gradient direction $\ttanv[\eta]_t =-\trgrad[x][t]$, and stepsize $\theta_t$), update by generating $\bar{x}_{t+1} := \dproj(\bar{x}_t  - \theta_t\bar{\eta}_t)$. 
The induced update rule on $\mrkk$ consists of the following steps: %
\begin{align}
&\tilde{X} := X - \theta (\grad f(X) - \theta \Gamma_X) \label{eq:rgdproj-pre} \\ 
&    X_{+} := \dproj(\tilde{X})=D_B^{(1)}\tilde{X}D_B^{(2)}, \label{eq:rgdproj-rule}
\end{align}
where $X:=\pi(\bar{x}_t)$, and $\grad f(X)$ and $\Gamma_X$ are defined in \Cref{lemm:rgd-mtotal2manf}.  

Given \Cref{prop:mc-rip}, we show that this normalized RGD rule enjoys some properties that maintain the validity of the convergence results on the $m\times n$ matrix space $\mrkk$ \Cref{ssec-ropt-mk:convana}. 

\begin{proposition} 
    \label{lemm:mc-proj}
Suppose that $X$ satisfies (i) $\fro{X-\mstar} \leq \epsilon_0 \sigma_k(\mstar)$, 
(ii) {$\rinfnorm[X-\mstar]\leq \epsilon_0 \rinfnorm[\mstar]$}, and (iii) $\sigma_{k}^{-1}(X)\sigma_k(\mstar) \leq \brk$ 
for constants $\epsilon_0>0$ and $\brk \lesssim 1$. %
Then, for $B\geq (1+\crinf) \sqrt{\mu\rkval}\sigmax$ and 
$\crinf:=(1+\brk)\theta\epsilon_0+\brk\theta^2\epsilon_0^2\lesssim\theta\epsilon_0(1+\theta\epsilon_0)$, 
the normalized RGD update $\dproj(\tilde{X})$~\eqref{eq:rgdproj-rule} given $X$ and stepsize $\theta$ satisfies: 
\begin{align}
    \max(\sqrt{m}\rinfnorm[{\dproj(\tilde{X})}], \sqrt{n}\rinfnorm[{\trs[{(\dproj(\tilde{X}))}]}]) \leq B & \quad\text{(same incoherence bound)}, \label{eq:sameinc}\\
    \fro{\dproj(\tilde{X}) -\mstar} \leq \fro{\tilde{X}-\mstar}  & \quad\text{(non-expansiveness)}, 
    \label{eq:nonexpa}
\end{align}
as long as $\max(\sqrt{m}\|X\|_{2,\infty}, \sqrt{n}\|\trs[X]\|_{2,\infty}) \leq B$. 
\end{proposition}

The following remark is about \cref{lemm:f-rsc-fort} to the
incoherence-restricted region $\nmg[{C_B,\mu}]$.

\begin{coro}
\label{thm:main-mc}    
Under the same assumptions of \cref{prop:mc-rip}: there exist constants $C_0$, $\cst$, $C_B$ such that for a sample size~\eqref{eq:mc-samplesize}, and an initial point $X_{0}\in\nmg[{C_B,\mu}]$~\eqref{def:nm-sl15}, 
the normalized RGD \eqref{eq:rgdproj-pre}--\eqref{eq:rgdproj-rule} ensures a sequence $
\{X_t\}_{t\geq 0}$ that converges to $\mstar$ with a linear rate. 
\end{coro}

This result for matrix completion, similar to the previous work, is based on a lemma (\Cref{lemm:mc-proj}) that ensures the RPD property. 
The difference, however, is that the convergence result in this subsection is built upon the generalized convergence result in \Cref{ssec-ropt-mk:convana}.

\section{Numerical experiments}\label{sec-mtns20:exp}

In this section, we conduct experiments to test the main algorithms (Section~\ref{ssec:opt-manf-via-mtotal}) on compressed sensing and matrix completion, which correspond to~\eqref{prog:main-mk} with the objective function~\eqref{eq:def-f-cs} and \eqref{eq:def-fit-mc} respectively. 
Below is a list of all tested algorithms. 

The main algorithms: the RGD (\Cref{algo-throptmk:generic-rgd}) is labeled as Qprecon~RGD. The term `Qprecon' signifies the preconditioned metric~\eqref{eq-throptmk:def-metric-precon} used in \Cref{algo-throptmk:generic-rgd} that induces the quotient metric on $\mrkk$. % 
Similarly, the RCG algorithm (Section~\ref{ssec:alg-rcg}) is labeled as Qprecon~RCG. %
The trial stepsize (\Cref{algo-throptmk:generic-rgd}, line~\ref{algline:qprecon-rgd-stepsize}) is computed using (i) exact line minimization (linemin), 
(ii) the Armijo linesearch method (Armijo), and (iii) the Riemannian Barzilai--Borwein (RBB) stepsize~\eqref{eq-throptmk:stepsize_bb}. % 
The RBB stepsize rule is also tested directly without backtracking linesearch; the underlying rule is labeled by `(RBB nols)'. 
 
The Euclidean algorithms: Euclidean gradient descent (Euclidean~GD) and nonlinear conjugate gradient (Euclidean~CG) refer to the GD and CG algorithms on the product space $\mtotal$ using the Euclidean gradients. %
The stepsize selection rules are the same as the main algorithms in Section~\ref{ssec:opt-manf-via-mtotal}. The Euclidean algorithms are implemented in the same package as the main algorithms. 

Existing algorithms on $\mrkk$: LRGeomCG~\cite{Vandereycken2013b}, NIHT and CGIHT~\cite{wei2016guarantees}, and ASD and Scaled-ASD~\cite{tanner2016low} are tested. %

The computation environment of all algorithms above is MATLAB. Given that $(m+n)\rkval\ll |\Omega|=\rho mn$ in low-rank problems, the dominant cost in all algorithms is the computation of the $m\times n$ residual matrix $S= P_{\Omega}(X - \mstar)$, % 
which costs $O(|\Omega|\rkval)$. The computation of $S$ in all algorithms is implemented in MEX functions.

All numerical experiments were performed on a workstation with 8-core Intel Core 549 i7-4790 CPUs and 32GB of memory running Ubuntu 16.04 and MATLAB R2015. 
The source code is available at \url{https://gitlab.com/shuyudong.x11/ropt_mk/}.
Implementations of the existing algorithms are also publicly available.

\subsection{Compressed sensing} 

The sensing operator $\Phi:\matsp\to\reals^p$ in~\eqref{eq:def-f-cs} is
represented by a matrix $\hat{\Phi}:\reals^{p\times mn}$ in which each row is
an i.i.d. Gaussian vector: $[\hat{\Phi}]_{ij} \sim \mc{N}(0,1)$ for $1\leq i
\leq p$ and $1\leq j\leq mn$. The sensing operator $\Phi$ is defined as
$\Phi(X) = \hat{\Phi}(\tvec[X])$. 
In the following test, the dimensions of the problem are $m=100$, $n=100$, and
$p/mn=0.25$. We generate a low-rank matrix $\mstar=G^\star \trs[H^\star]$, where
$G^\star$ and $H^\star$ are thin Gaussian matrices of size $m\times r$ and
$n\times r$ respectively, with $r=5$. The observed vector from the sensing
operator is $b^{\star} = \Phi(\mstar)\in\reals^{p}$. 
We test the algorithms using the initial point $X_0:= \Phi^{*} (b^\star)$. 
The performances of Qprecon~RGD (Armijo and RBB) and Euclidean GD (Armijo) are shown in \cref{fig:exp-cs-rperi}. % 

\begin{figure}[htpb]
    \centering
        \includegraphics[width=.42\textwidth]{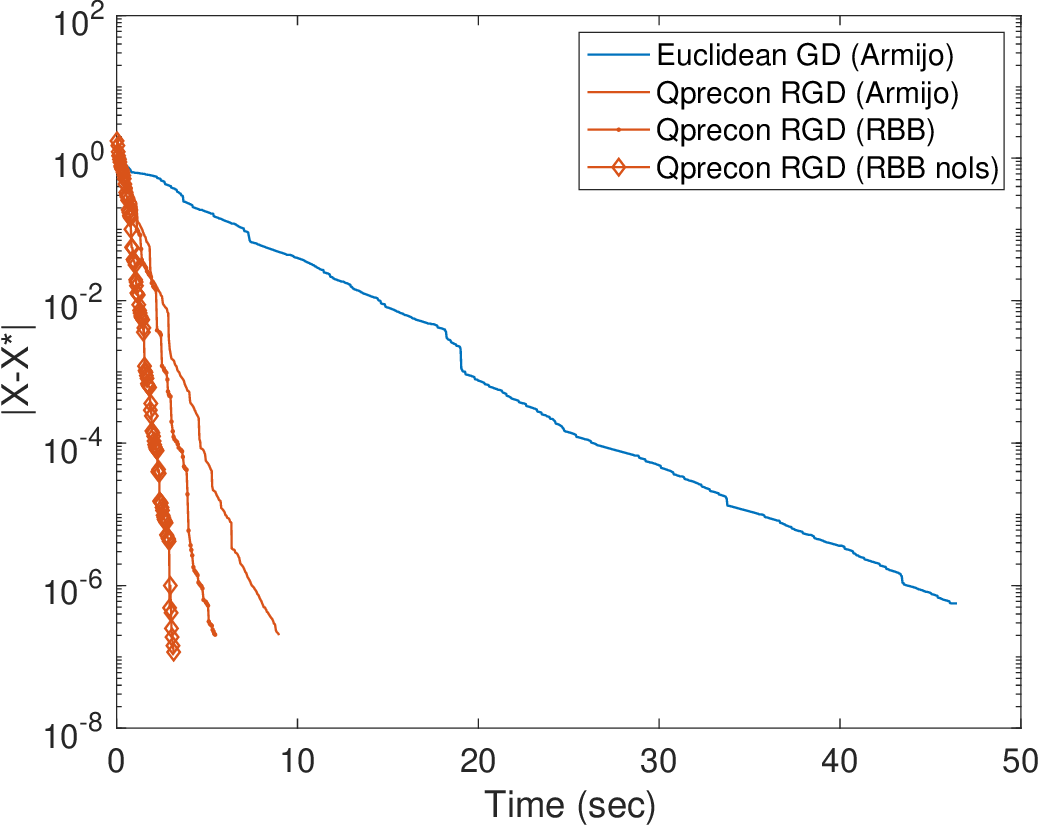}\qquad\quad
        \includegraphics[width=.42\textwidth]{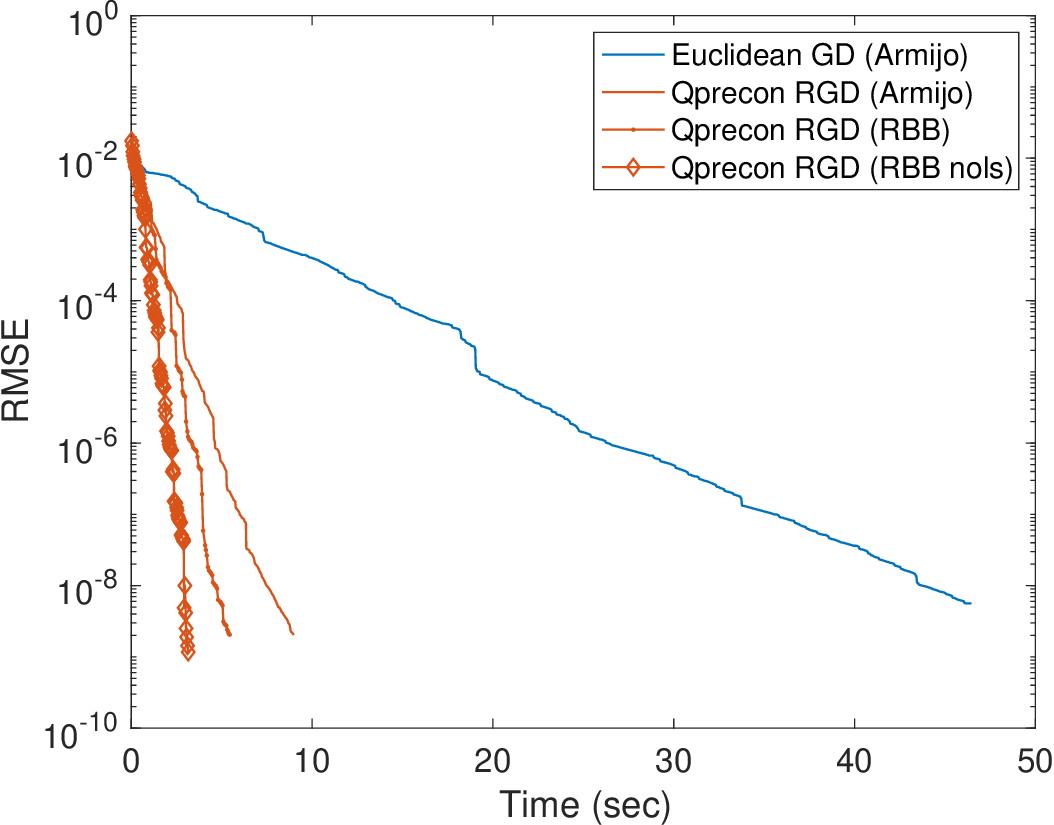}
    \caption{Compressed sensing results. }
    \label{fig:exp-cs-rperi}
\end{figure}

We observe from \cref{fig:exp-cs-rperi} that Qprecon~RGD, with both the Armijo and RBB rules, are faster than Euclidean GD by orders of magnitude. 
From the iteration history in~\cref{fig:exp-cs-rperi}, we also observe that the decay of the residual $\fro{X_{t} - \mstar}$ is linear, which agrees with the result given in \cref{lemm:f-rsc-fort}.

In particular, the RBB stepsize rule~\eqref{eq-throptmk:stepsize_bb} is tested both with and without backtracking linesearch under the same experimental setting. %
\cref{fig:exp-cs-bb} shows the iteration history of the gradient norms and stepsizes of these two variants (RBB and RBB nols). %
In \cref{fig:exp-cs-bb}, % 
$\text{SS}=\tmet[x][t](\bar{z}_{t-1},\bar{z}_{t-1})$, 
$\text{SY}= \tmet[x][t](\bar{z}_{t-1},\bar{y}_{t-1})$, and 
$\text{YY}=\tmet[x][t](\bar{y}_{t-1},\bar{y}_{t-1})$ are the values appearing in the RBB formula~\eqref{eq-throptmk:stepsize_bb}. % 

\begin{figure}[htbp]
    \centering
\subfigure[with backtracking linesearch]{\includegraphics[width=.42\textwidth]{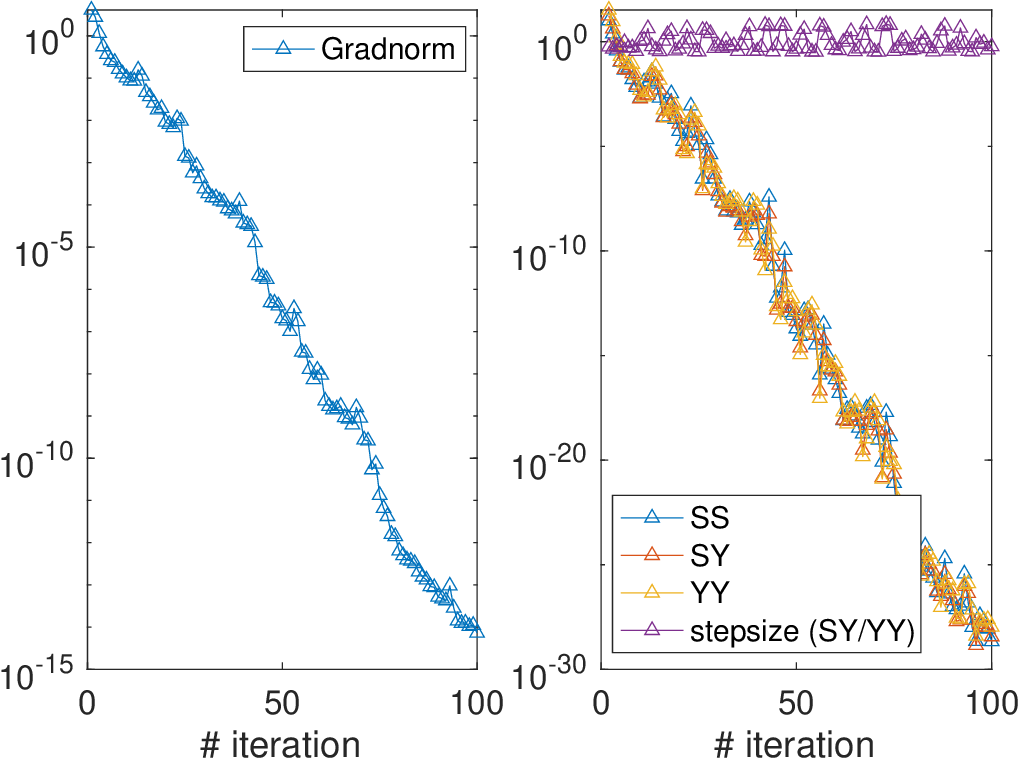}} 
\qquad\quad 
\subfigure[without backtracking linesearch]{\includegraphics[width=.42\textwidth]{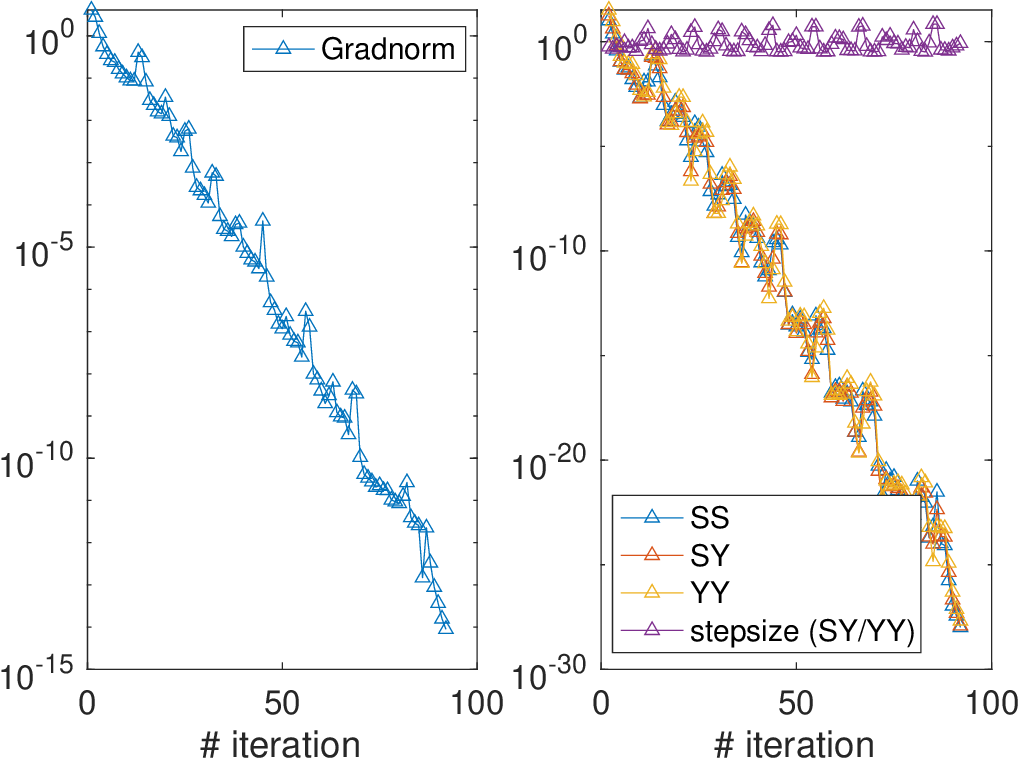}} 
\caption{Gradient and stepsize information of Qprecon RGD (RBB). }
    \label{fig:exp-cs-bb}
\end{figure}

We observe from \cref{fig:exp-cs-bb} that the decay of the gradients' norm (`Gradnorm') is linear, and the RBB stepsizes (`SY/YY', in purple) are rather stable, fluctuating around $O(1)$ with small variation. The RBB stepizes (both with and without linesearch) that enable the convergence agree with the characterization in \Cref{lemm:f-rsc-fort} (stating that there exists $\ssz^*>0$ such that the algorithm converges with a linear rate, for all stepsizes $\theta_t< \ssz^*$). 

\subsection{Matrix completion}\label{sec:exp-invpath}

In the following experiments, the partially observed matrix $\mstar\in\vecsp$ is generated as follows: 
for a rank parameter $\rkval\ll \min(m,n)$, $\mstar = A^{\star} B^{\star}$, 
where $A\in\reals^{m\times \rkval}$ and $B\in\reals^{m\times \rkval}$ are
composed of entries drawn from the Gaussian distribution $\mc{N}(0,1)$.  
The index set of the observed entries (training data) consists of indices sampled from the Bernoulli distribution $\mc{B}(p)$: for $(i,j)\in\iset{m}\times\iset{n}$, 
\begin{equation}\label{def:Omega-bernoulli}
    (i,j)\in\Omega, \mathrm{~with~probability~} p.
\end{equation}

\paragraph{Optimization paths on \texorpdfstring{$\mrkk$}{Mk}}\label{parag:exp-invpath}
\cref{lemm:rgd-mtotal2manf} shows that Qprecon~RGD (\Cref{algo-throptmk:generic-rgd}), by producing a sequence $\algseq\subset\mtotal$, induces a sequence $\indseq$ on $\mrkk$ along the negative Riemannian gradients in $\tansp[]\mrkk$. This means that the sequence does not depend on the locations of the iterates $\algseq$ in the vertical space $\tversp$. 
To demonstrate this property, we observe next the iteration histories and paths of the induced sequences of Qprecon~RGD and Euclidean GD on $\mrkk$.

\begin{figure}[htbp]
\centering
\subfigure[Armijo line search]{\includegraphics[width=.43\textwidth]{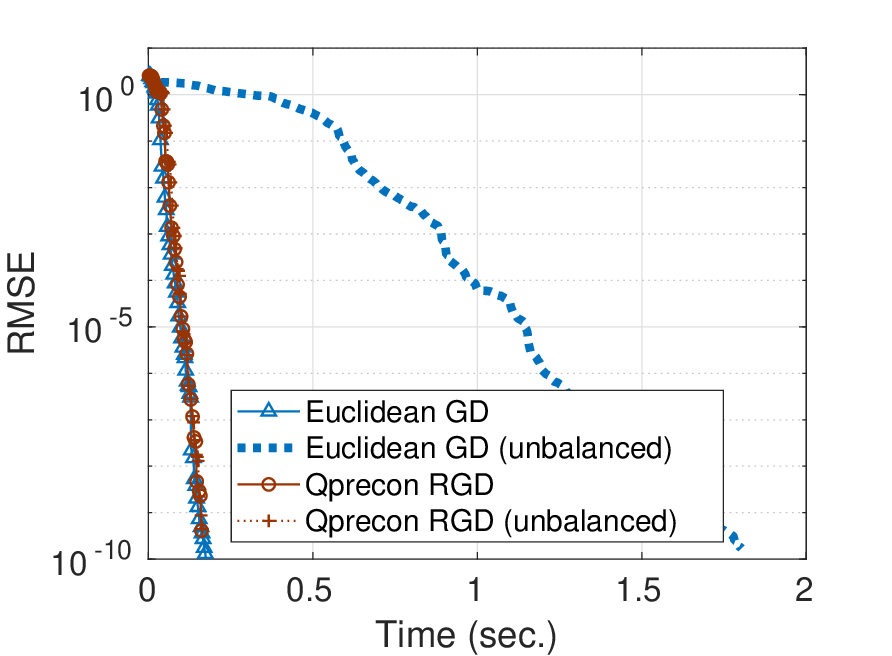}} \qquad 
\subfigure[Line minimization]{\includegraphics[width=.43\textwidth]{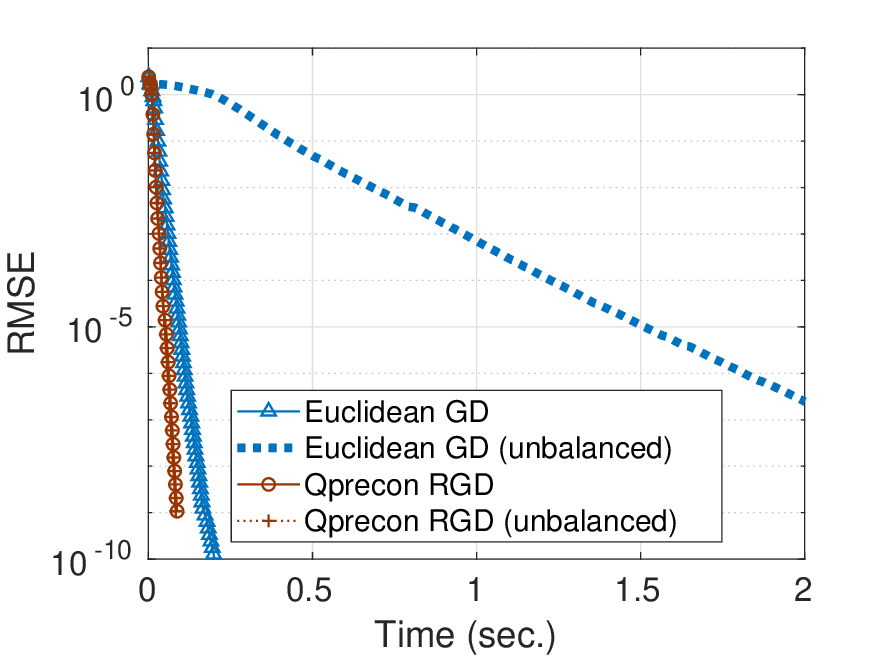}}\\ 
\subfigure[Balanced]{\includegraphics[width=.43\textwidth]{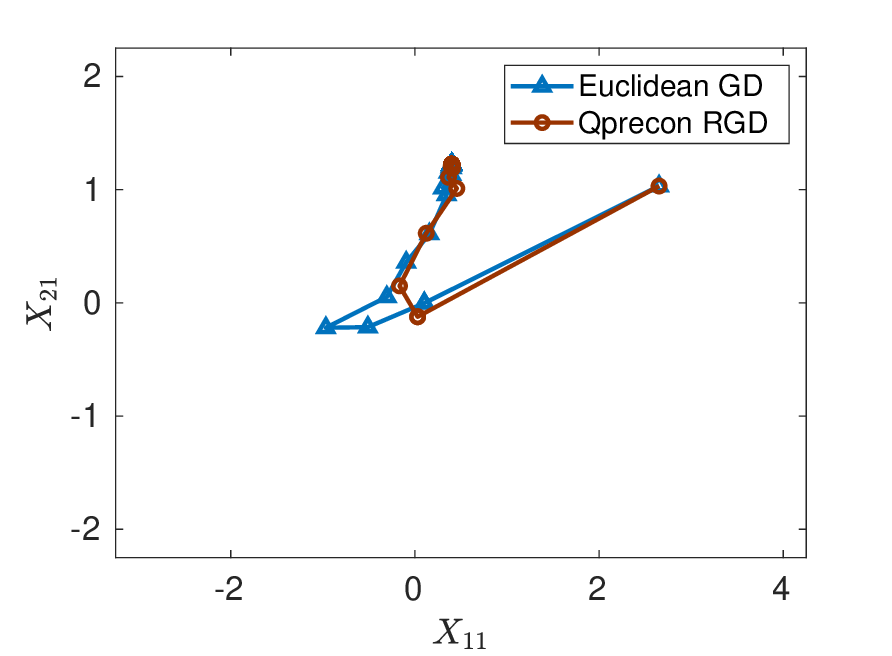}}\qquad 
\subfigure[Unbalanced]{\includegraphics[width=.43\textwidth]{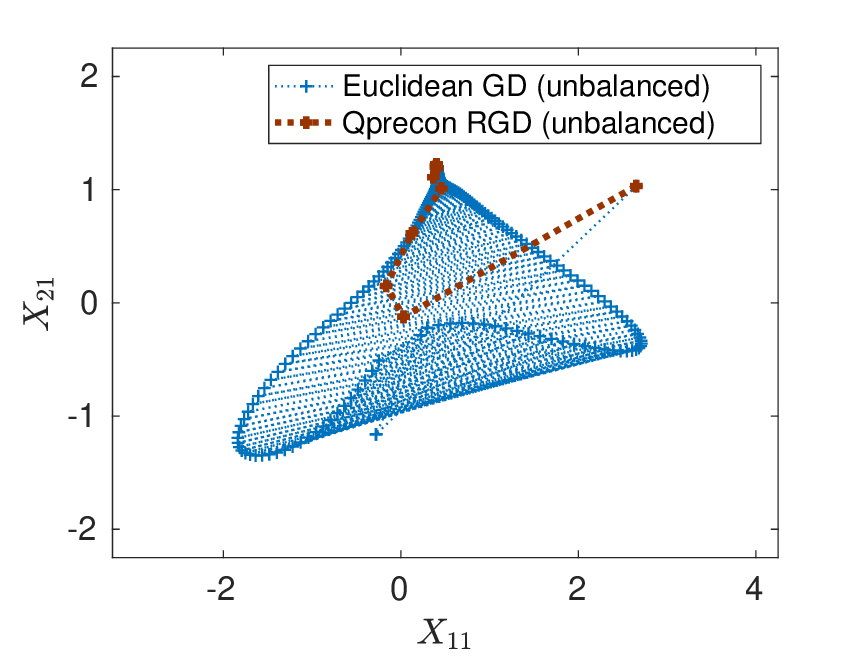}} 
\caption{Iteration histories of the algorithms with a balanced and an unbalanced
initial point. The size of $\mstar$ is $100\times200$, with rank $\rstar=3$. The
sampling rate $p=0.8$.
(a)--(b): test RMSEs by time. (c)--(d): paths of the matrix entries $([X_t]_{1,1}, [X_t]_{2,1})$ of the iterates $\{\pi(\bar{x}_{t})\}$.}
\label{fig:cropt-mk-invpath}
\end{figure}

The matrix completion tests are conducted on noiseless observations of $\mstar$ on $\Omega$~\eqref{def:Omega-bernoulli} with a sampling rate $p=0.8$. 
Both Qprecon~RGD and Euclidean GD are initialized with the same initial point in the following two settings: % 
(i) the {\it balanced} setting: each algorithm is initialized with $x_0 = (G_0,H_0)$ using the spectral initialization method such that $\|G_0\|=\|H_0\|$, and (ii) the {\it unbalanced} setting: the initial point is defined as 
$x_0' = (\lambda G_0, H_0/\lambda)$ for $\lambda = 5$, where $G_0$ and $H_0$ are the same as in the balanced setting. The comparative results are given in Figure~\ref{fig:cropt-mk-invpath}.

From \Cref{fig:cropt-mk-invpath}, we observe that the two sequences of Qprecon~RGD overlap, which shows indeed that the path of the sequence generated by Qprecon~RGD does not vary with the change in the initial point. We also observe that these overlapping sequences converge linearly, with much faster speed than Euclidean GD with the unbalanced initial point. 
In fact, one can see from the figure that the convergence of Euclidean GD is
significantly slowed down with the unbalanced initial point $x_0'$ compared to
the case with $x_0$.

\paragraph{Comparisons with existing algorithms}\label{ssec-mtns20:rperi-comparisons}
We compare the performances of Qprecon RGD and Qprecon RCG with the other listed algorithms~\cite{Vandereycken2013b,wei2016guarantees,tanner2016low}.
More precisely, LRGeomCG~\cite{Vandereycken2013b}, NIHT and CGIHT~\cite{wei2016guarantees} are based on the same manifold structure of $\mrkk$ and the same retraction operator. 
These algorithms produce iterates % 
on the product space
$\stie(m,\rkval)\times\mathcal{S}_{++}(\rkval)\times \stie(n,\rkval)$ using
gradients defined with the Euclidean metric. A retraction is needed to ensure the fixed-rank constraint, %
for which the projection-like retraction~\cite{Absil2012} (also called the IHT for ``iterative hard thresholding'') is used. % 
ASD and ScaledASD~\cite{tanner2016low} use alternating steepest
descent on $\prodspstar$. %

\begin{figure}[htbp]
\centering
\subfigure[$\rstar=10$]{\includegraphics[width=.31\textwidth]{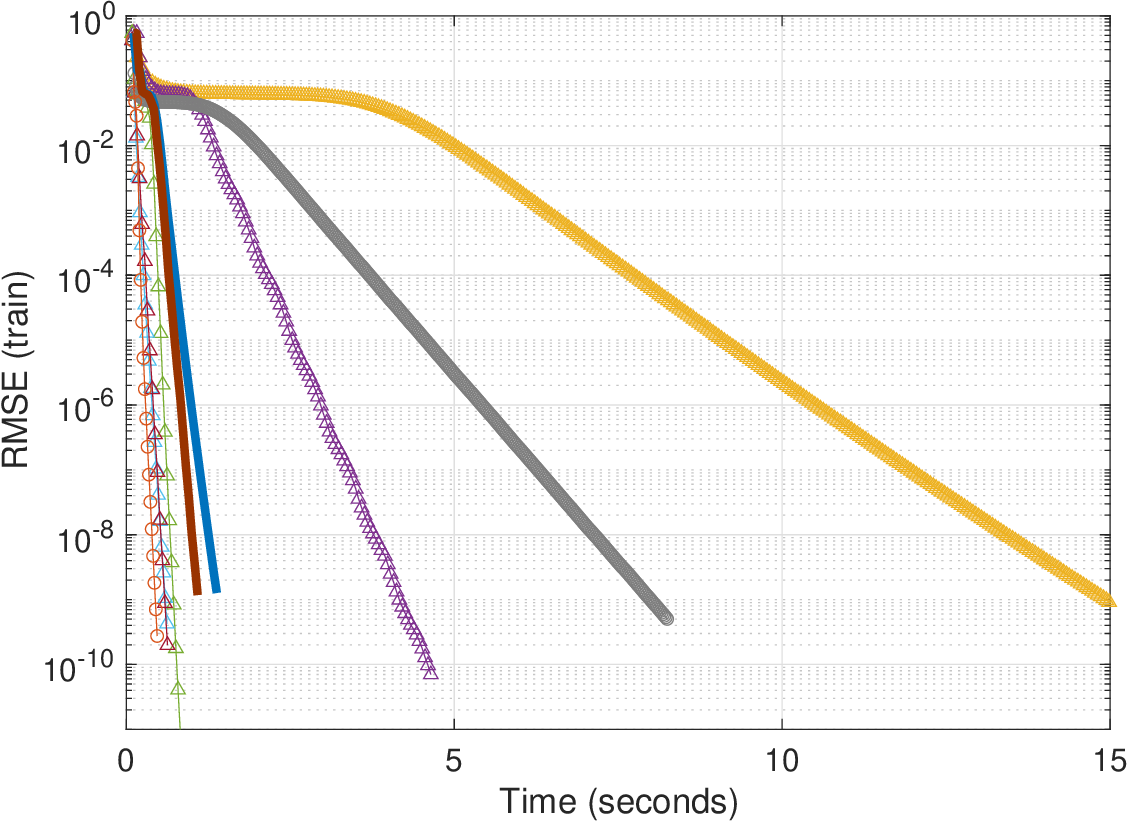}}
~
\subfigure[$\rstar=20$]{\includegraphics[width=.32\textwidth]{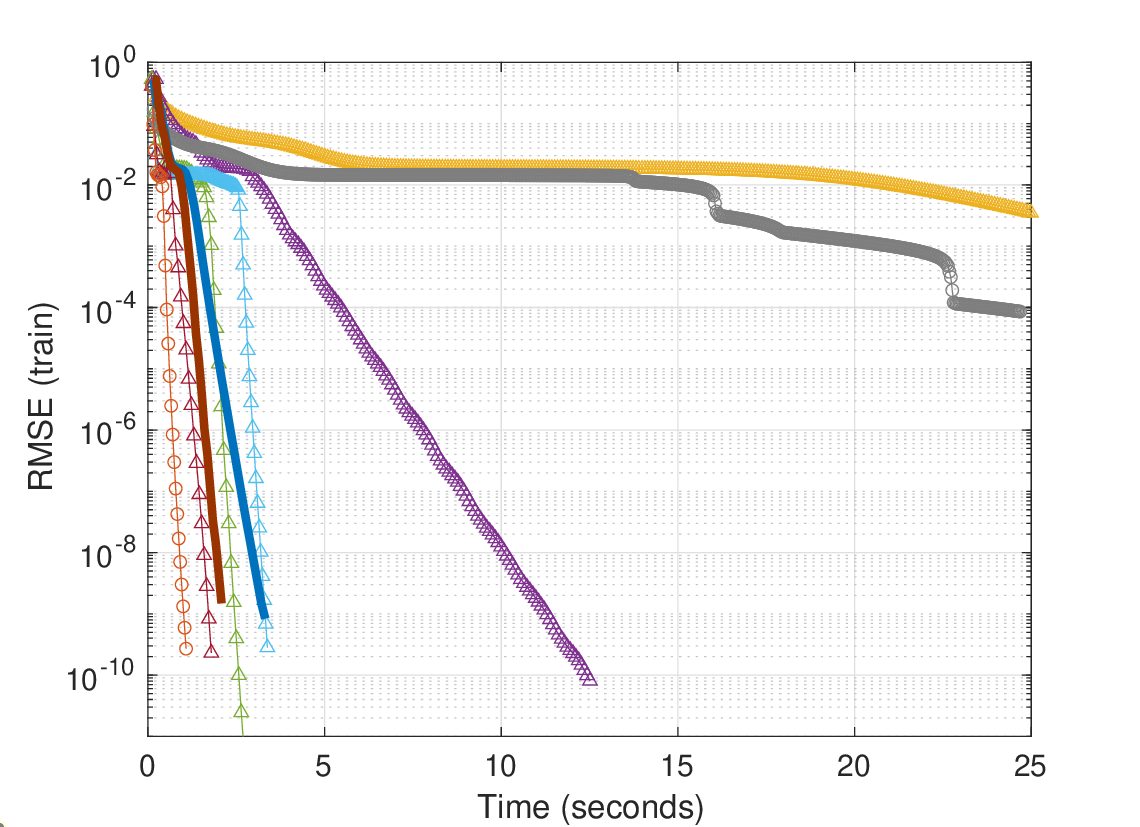}} 
~
\subfigure[$\rstar=30$]{\includegraphics[width=.32\textwidth]{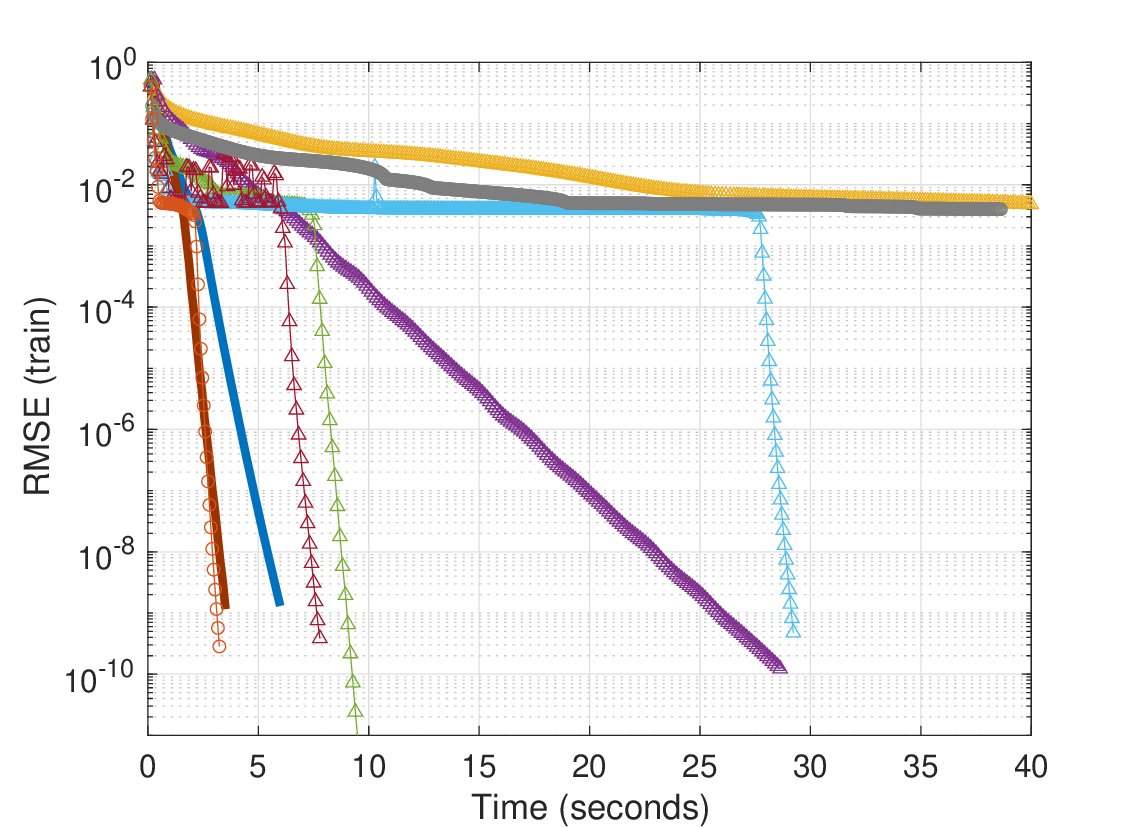}} 
\\
\subfigure[legend of (a-c)]{\includegraphics[width=.29\textwidth]{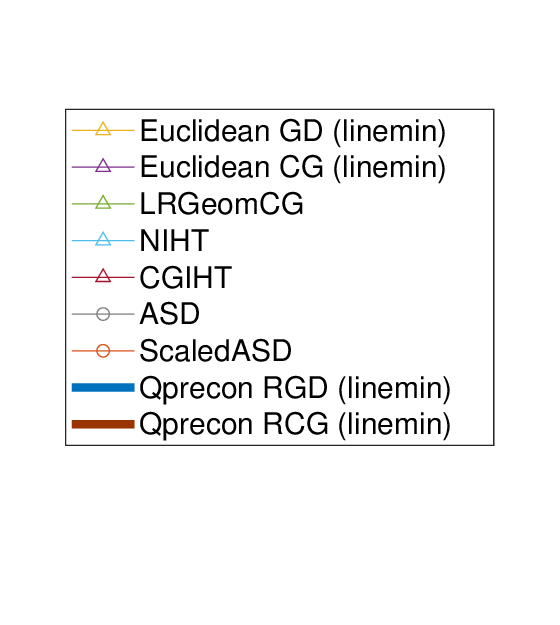}} 
\qquad\qquad  
\subfigure[Average time per-iteration]{\includegraphics[width=.38\textwidth]{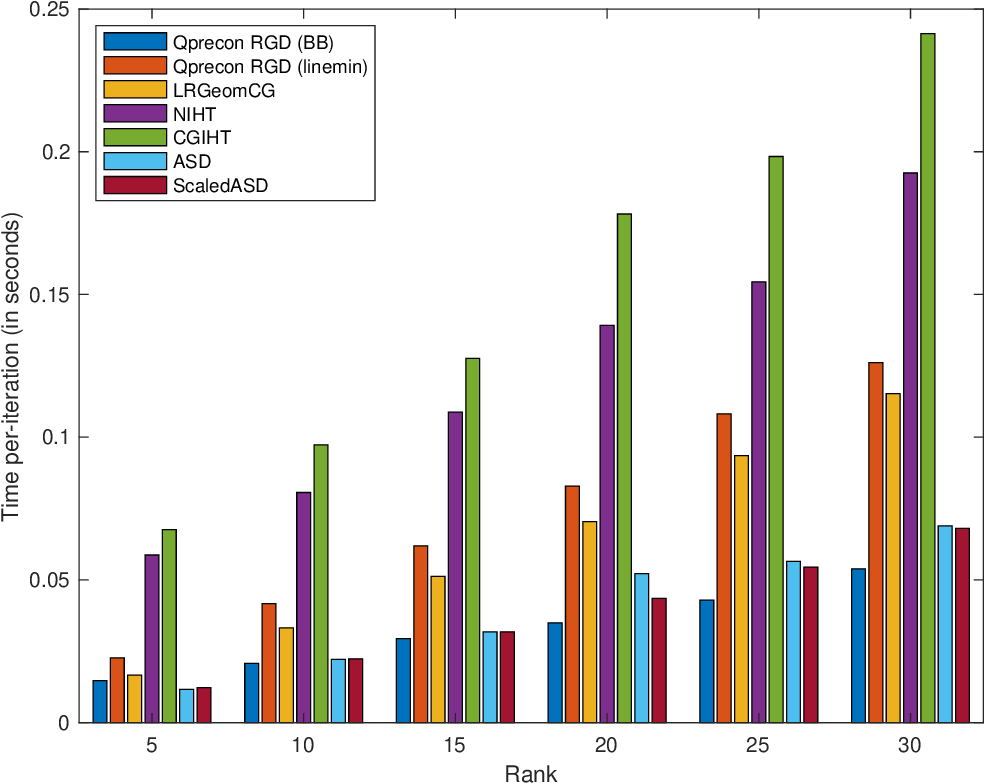}}
\caption{Performances of algorithms on matrix completion. The matrix size is $800\times900$ with sampling rate $p=0.6$ for all tests. The rank parameter $\rkval=\rstar$ in each setting. % 
(a-d): iteration history under rank values $\rstar\in\{10, 20\}$. %
(e): average time per-iteration with ranks varying in $\{5, 10, \dots, 30\}$. 
}\label{fig:eid23mar-n0rstar-1}
\end{figure}

In \Cref{fig:eid23mar-n0rstar-1}, we show the performances of the algorithms in recovering low-rank matrices from the same matrix model as the previous experiment. % 
From \Cref{fig:eid23mar-n0rstar-1}~(a-d), we observe that our algorithms (Qprecon~RGD and Qprecon~RCG) %
perform similarly as ScaledASD and they are all either faster than or comparable to
the rest of the algorithms under all three different rank values. %
We also observe that algorithms using Riemmaian geometry on $\mrkk$---ours, scaledASD, LRGeomCG, and NIHT/CGIHT---outperform those (Euclidean GD/CG, ASD) based on Euclidean geometry, with only one exception, where NIHT attained stationarity a bit later than Euclidean CG when $r^\star=30$. In particular, when $r^\star=10$ and all algorithms attain stationarity, both our algorithms, scaledASD and LRGeomCG achieve a speedup of around $10$ times over Euclidean GD; and when $r^\star \in \{20,30\}$, Euclidean GD and ASD exceeded the maximal time budget before attaining the target stationarity precision. 
From \Cref{fig:eid23mar-n0rstar-1} (e), we observe that Qprecon~RGD (RBB), ASD and Scaled ASD are the three most efficient algorithms in terms of per-iteration cost, with all different rank values varying in $\{5, 10, \dots, 30\}$. In particular, these three algorithms have better scalability in the rank $\rkval$ than the rest of the algorithms.

%%%
\section{Conclusion}\label{sec-throptmk:conclusion}
We investigated methods for the matrix completion problem with a fixed-rank
constraint and focused on a Riemannian gradient descent algorithm in the framework of
optimization on the quotient manifold of fixed-rank matrices.
We showed that the Riemannian gradient descent algorithm under the aforementioned quotient geometric setting not only enjoys the advantage of matrix factorization but can also be analyzed in a more convenient way than existing matrix factorization methods.
We developed novel results for analyzing the quotient
manifold-based algorithm and proved that this algorithm solves the fixed-rank matrix completion
problem with a linear convergence rate. Moreover, the convergence property of
the algorithm has desirable invariance properties in contrast to Euclidean
gradient descent algorithms. 
Because of the efficient iteration efficiency and its light per-iteration cost, the time efficiency of this algorithm is also shown to be much faster than the Euclidean gradient descent algorithms and is faster than many other Riemannian algorithms on the set of fixed-rank matrices.
Through the convergence analysis, we also provided a novel understanding of the
graph-based regularization in the theoretical framework of matrix completion.

%%%%%%%%%%%%%%

\section*{Acknowledgments}
We gratefully acknowledge the valuable comments and suggestions from {P.-A.} Absil during the development of the main results of this paper. %

\section*{Funding}
This work was supported by the Fonds de
la Recherche Scientifique -- FNRS and the Fonds Wetenschappelijk Onderzoek --
Vlaanderen under EOS Project no.\ 30468160 \revj{(SeLMA -- Structured low-rank matrix/tensor
approximation: numerical optimization-based algorithms and applications)}. The
first author was supported by the FNRS--FRIA scholarship at UCLouvain, Belgium during this work. % 
Bin Gao was supported by the Deutsche Forschungsgemeinschaft (DFG, German Research Foundation) via the collaborative research centre 1450--431460824, InSight, University of M\"unster, and via Germany's Excellence Strategy EXC 2044--390685587, Mathematics M\"unster: Dynamics--Geometry--Structure.

\appendix
\section{Proofs in \Cref{sec:opt-on-mk}} 

\begin{proof}[{\bf Proof of \Cref{prop:invariance-metric-precon}.}] 
    For any $\bar{x}:=(G,H)\in\mtotal$ and $\bar{x}'\sim\bar{x}$, there exists an invertible matrix
    $F\in\glin[\rkval]$, such that $\bar{x}'=(G\trs[F], HF^{-1})$. 
    Since the tangent vector $\ttanv[\xi']$ %
    must satisfy 
    $\dop\pi(\bar{x}')[\ttanv[\xi']] = \dop\pi(\bar{x})[\ttanv]$, and
    note that $\dop\pi(\bar{x})[\ttanv]=G \ttanvr + \ttanvl \trs[H]$ for $\bar{x}=(G,H)$, $\ttanv'$ and $F$ must satisfy 
    \[G \ttanvr + \ttanvl \trs[H] = G\trs[F] \ttanvr[\xi']  +
    \ttanvl[\xi'] \trs[(HF^{-1})],\]
    for any $\ttanv\in\ttansp$, which yields 
    \[ \ttanvl[\xi'] = \ttanvl \trs[F]\mathrm{~and~}  \bar{\xi'}^{(2)} = {\bar{\xi}}^{(2)} F^{-1}.
    \]
    The same relation holds for $\ttanv[\eta], \ttanv[\eta']$ and $F$. 
    Applying these equalities into the expression of 
    $\bar{g}_{\bar{x}'}(\ttanv',\ttanv[\eta]')$, where $\bar{g}$ is the
    preconditioned metric~\eqref{eq-throptmk:def-metric-precon}, we recover immediately the expression of $\bar{g}_{\bar{x}}(\ttanv,\ttanv[\eta])$. 
\end{proof}

\section{Proofs in \Cref{sec-ropt-mk:lrqp}}

\begin{proof}[{\bf Proof of \Cref{prop:rgrad-total2quot}.}] 
Let $\trgrad=(\ttanvl[\eta], \bar{\eta}^{(2)})$ denote the two components of
the gradient~\eqref{eq-throptmk:rgrad-man-precon} of $\bar{f}$,  
where $\bar{x}:=(G,H)\in\pi^{-1}(X)$ given $X=U\Sigma\trs[V]\in\mrkk$. 
We have 
 	\[
 	\ttanvl[\eta]=\tfpartg(\trs[H] H)^{-1}\text{~and~}
    \bar{\eta}^{(2)}=\tfparth(\trs[G] G)^{-1},
 	\] where
 	$\tfpartg = \nabla f(X)H$ and $\tfparth=\trs[\nabla f(X)] G$. 
    \bleu{Since $\trgrad$ is the horizontal lift of $\rgrad$, by definition,
    $\rgrad=\dop\pi(\bar{x})[\trgrad]$.} 
    Therefore, we have  
\begin{align*}%
 	\rgrad&= \dop\pi(\bar{x})[\trgrad] 
 		= G\ttanvr[\eta] + \ttanvl[\eta]\trs[H]\\ 
        &= G(\trs[G] G)^{-1}\trs[G]{\nabla f(X)}+ \nabla f(X) H(\trs[H] H)^{-1}H\\
 		&= \pu\nabla f(X) + \nabla f(X)\pu[V],
\end{align*}
where $\pu$ and $\pu[V]$ are the matrices defined in the statement.
\end{proof}

\begin{proof}[{\bf Proof of \Cref{lemm:rgd-mtotal2manf}.}] 
To simplify the notations, we omit the subscript $t$ of all terms related to the
$t$-th iterate ($G_{t}$, ${\ttanv[\eta]}_{t}$, stepsize $\ssz_t$) and denote
$X_{t+1}$ by $X_{+}$ in the following equations. Let
${\ttanv[\eta]}:=\trgrad[x][]$ denote the Riemannian gradient of $\bar{f}$ on
the current iterate $\bar{x}:=(G,H)$. 
From the update step in Algorithm~\ref{algo-throptmk:generic-rgd}, we have
\begin{align}%
X^{+} &:=\pi(\bar{x}-\ssz\trgrad) %
 = \pi(\bar{x}) - \ssz\left(G \ttanvr[\eta] + \ttanvl[\eta] \trs[H]\right) +
    \ssz^{2}\ttanvl[\eta] \ttanvr[\eta]\notag~\\
 &= X -\ssz \rgrad + \ssz^{2}\partial_{G}\bar{f}(\bar{x})(\trs[H]H)^{-1}
    (\trs[G] G)^{-1} \trs[(\partial_{H}\bar{f}(\bar{x}))]\label{eq:prf-rgd-mtotal2manf-l1}\\
 &= X -\ssz \rgrad + \ssz^{2}\underbrace{\nabla f(X)H(\trs[H]H)^{-1}
    (\trs[G] G)^{-1} \trs[G] \nabla f(X)}_{\phi_1}\label{eq:l4-rgd-mtotal2manf}~\\
 &=  X -\ssz \rgrad + \ssz^{2}\nabla f(X)X^{\dagger}\nabla f(X),\label{eq:prf-rgd-mtotal2manf-lend}
\end{align}
which proves~\eqref{eq:lemm-rgd-mtotal2manf}. 
The equation~\ref{eq:prf-rgd-mtotal2manf-l1} is obtained by identifying 
$(G\ttanvr[\eta]+ \ttanvl[\eta] \trs[H])$ with 
\[
\dop\pi[(G,H)](\ttanv[\eta]):=\dop\pi[(G,H)](\trgrad)=\rgrad[X].
\] 
The equation~\eqref{eq:prf-rgd-mtotal2manf-lend} is obtained as follows. Let
$X:=U\Sigma \trs[V]$ denote the $\svd$ of $X\in\mrkk$, where where
$U\in\stie(m,\rkval)$ and $V\in\stie(n,\rkval)$ and
$\Sigma\in\reals^{\rkval\times\rkval}$. 
Then we use the fact that there
exists $F\in\glin[\rkval]$, for any $(G,H)\in\pi^{-1}(X)$, such that
$G=U\Sigma_{G}\trs[F]$ and $H=V\Sigma_{H}F^{-1}$, where $\Sigma_{G}$ and $\Sigma_{H}$ are $\rkval\times \rkval$ diagonal matrices such that $\Sigma_{G}\Sigma_{H} = \Sigma$. Therefore, $\phi_1$ in the last term of the right-hand side of~\eqref{eq:l4-rgd-mtotal2manf} reads:
\begin{align*}
\phi_1 & =  \nabla
f(X)V\Sigma_{H}F^{-1}(F\Sigma_{H}^{-2}\trs[F])(F^{-T}\Sigma_{G}^{-2}F^{-1})F\Sigma_{G}\trs[U]\nabla f(X)~\\
&= \nabla f(X)V\Sigma_{H}(\Sigma_{H}^{-2}\Sigma_{G}^{-2})\Sigma_{G}\trs[U] \nabla f(X)~\\
&= \nabla f(X)V\Sigma_{H}^{-1}\Sigma_{G}^{-1}\trs[U] \nabla f(X)~\\
&= \nabla f(X)X^{\dagger}\nabla f(X).
\end{align*}
\end{proof}

\begin{proof}[{\bf Proof of \Cref{lemm:opvx-props}.}] 
(i): From the definition~\eqref{eq:ov-egrad2rgrad}, 
$\opvx(Z)=\pu(Z)+ \pu[V](Z)$, we deduce 
that $\opvx$ is a symmetric operator, since the orthogonal projections $\pu(\cdot)$ and $\pu[V](\cdot)$ are symmetric operators. We prove the remaining claims as follows. 
From~\eqref{eq:rmk-rgrad-vs-proj} in \cref{rmk:ov-egrad2rgrad}, 
we deduce that $\opvx(Z)=\proj_{\tansp[X]\mrkk}(Z) + U\trs[U]Z V\trs[V]=2Z$ if $Z\in\tansp[X]\mrkk$. 
If $Z'\in(\tansp[X]\mrkk)^{\perp}$, then 
there exists $Y\in\vecsp$ such that $Z'=(\id-\proj_{\tansp[X]\mrkk})(Y)$, and consequently 
\begin{align*}
	\opvx[X](Z')&=\proj_{\tansp[Y]\mrkk}((\id-\proj_{\tansp[X]\mrkk})(Y)  )+U\trs[U] ((\id-\proj_{\tansp[X]\mrkk})(Y)) V\trs[V] 
	= 0. 
\end{align*}
Therefore, for any $Z\in\vecsp$, $\opvx(Z)=\opvx(\proj_{\tansp[X]\mrkk}(Z))=2\proj_{\tansp[X]\mrkk}(Z)$, which entails that $0\leq2\fro{\proj_{\tansp[X]\mrkk}(Z)}^2=\braket{\opvx(Z),Z}\leq 2\frobb[Z]$. 

(ii): %
The matrix $Z:=Y-X$ can be decomposed as $Z=\tilde{Z} + \Delta_Z$, where 
$\tilde{Z}:=\proj_{\tansp[X]\mrkk}(Z)$ and %
$\Delta_Z: =(\id-\proj_{\tansp[X]\mrkk})(Z)=(\id-\proj_{\tansp[X]\mrkk})(Y)$, where the last equality holds since $(\id-\proj_{\tansp[X]\mrkk})(X)=0$. %
Therefore, we have 
\begin{align}
&(\opvx-2\id)(X-Y) %
= -(\opvx-2\id)(\tilde{Z}+\Delta_Z)\notag\\
&= (2\id-\opvx)(\Delta_Z)%
= 2\Delta_Z := 2(\id-\proj_{\tansp[X]\mrkk})(Y), \label{eq:iddiff-l2} 
\end{align}
where the equalities in~\eqref{eq:iddiff-l2} are obtained by using the fact that $\tilde{Z}\in\tansp[X]\mrkk$ and $\Delta_{Z}\in(\tansp[X]\mrkk)^{\perp}$ and the properties ${\opvx}_{|\tansp[X]\mrkk}=2\id$ and ${\opvx}_{|(\tansp[X]\mrkk)^{\perp}}=0$, proven in (i). 
\end{proof}

\section{Proofs in \Cref{ssec-ropt-mk:lrqp-intro}} 
\label{sec-app:prf-sec4}

\subsection{Proofs of lemmas in~\Cref{ssec-ropt-mk:convana}}
\label{ssec-ropt-mk:prfs-mainres}

\begin{proof}[{\bf Proof of \cref{lemm:f-hess-pd}.}] 
    \label{parag:prf-hess-pd} 
The Riemannian Hessian of $f$ at $Y\in\mrkk$, by definition~\cite[Proposition 5.5.4]{AbsMahSep2008}, satisfies 
\begin{equation}
    \label{eq:def-hessf-lrqp}
    \braket{\hess f(Y)[Z],Z} = \left.\ddtt(f(\expm_{Y}(tZ)))\right|_{t=0},
\end{equation}
where $\expm_{Y}:\tansp[Y]\mrkk\to\mrkk$ is the exponential map at $Y$. %
We prove the spectral lower bound of $\hess f(\mstar)$ as follows. 
First, the exponential map at $Y:=\mstar$ %
has the following expression (\eg,~\cite{Absil2012},~\cite[Proposition A1]{Vandereycken2013b},~\cite[Appendix A]{10.1093/imanum/drz061}), 
\begin{equation}
\expm_{\mstar}(Z) = \mstar + Z  + \dttx[\mstar](Z), \label{eq:def-expm-mk} % 
\end{equation}
where $\dttx[\mstar](Z):= (I-\pu[U^\star])Z {\mstar}^{\dagger} Z(I-\pu[V^{\star}]) + O(\fro{Z}^3)$ 
satisfies $\fro{\dttx[\mstar](Z)}\lesssim \fro{Z}^2$. 
From~\eqref{eq:def-hessf-lrqp}--\eqref{eq:def-expm-mk}, it follows that the quadratic function $f$~\eqref{eq:def-f-lrqp} satisfies, for any $Z\in\tansp[\mstar]\mrkk$, 
\begin{equation}
    \label{eq:hessf-lrqp}
    \braket{\hess f(\mstar)[Z],Z} = \braket{\opa(Z),Z}.
\end{equation}

Next, we bound~\eqref{eq:hessf-lrqp} using %
\Cref{prop-rpd2asp}. Consider $\tilde{f}: X\mapsto f(X) - f(\mstar)$ 
where $f$ is defined in~\eqref{eq:def-f-lrqp}, we have
\begin{align}% 
\tilde{f}(X) &= \frac{1}{2} \braket{\opa(X),X} - \braket{\bstar,X} -
f(\mstar)\notag\\
&= \frac{1}{2}\braket{\opa(X),X}-\braket{\opa(\mstar),X}-(-\frac{1}{2}\braket{\opa(\mstar),\mstar})\label{eq:barf-l1}\\
& = \frac{1}{2}\braket{\opa(X-\mstar),X-\mstar}\geq \frac{1-\brip}{2}\fro{X-\mstar}^2, \label{eq:barf-lend} 
\end{align}
where the equality~\eqref{eq:barf-l1} holds due to the fact that $\bstar=\opa(\mstar)$ by definition~\eqref{eq:def-f-lrqp}, the equality in~\eqref{eq:barf-lend} holds since $\opa$ is a symmetric operator, and the last inequality is a direct result of 
\Cref{prop-rpd2asp}.
It follows that, for any $Z\in\tansp[\mstar]\mrkk$ %
and $X:=\expm_{\mstar}(Z)\in\mrkk$, 
\begin{equation}\label{eq:prf-hess-ineq}
 \tilde{f}(\expm_{\mstar}(Z)) \geq \frac{1-\brip}{2}\frobb[\expm_{\mstar}(Z)-\mstar], 
\end{equation} 
which can be rewritten %
via the expression~\eqref{eq:def-expm-mk} of $\expm_{\mstar}(Z)$ as follows, 
\begin{equation}\label{eq:prf-hess-ineq2}
\tilde{f}(\expm_{\mstar}(Z))= \braket{\opa(Z),Z}  +  \varphi_{\tilde{f}}(Z) \geq (1-\brip)\frobb[Z] + \varphi_{F}(Z),
\end{equation}
where $\varphi_{\tilde{f}}(Z)$ and $\varphi_{F}(Z)$ are the sums of third- and
higher-order terms of $Z$ on the two sides of~\eqref{eq:prf-hess-ineq}, %
\ie, %
$|\varphi_{\tilde{f}}(Z)| \leq C_1 \fro{Z}^{3}$ and %
$|\varphi_{F}(Z)| \leq C_2 \fro{Z}^{3}$, 
for constants $C_1>0$ and $C_2>0$. 

Combining~\eqref{eq:prf-hess-ineq2} and~\eqref{eq:hessf-lrqp}, we have, for any $Z\in\tansp[\mstar]\mrkk$, %
\begin{align*}
    \lambda_{f}(Z)&:=\frac{\braket{\hess f(\mstar)[Z],Z}}{\frobb[Z]} = \frac{\braket{\opa(Z),Z}}{\fro{Z}^2} 
    \geq (1-\brip)
    -\frac{|\varphi_{\tilde{f}}(Z)|+|\varphi_{F}(Z)|}{\frobb[Z]} \notag\\
    &\geq (1-\brip) - \left(C_1+C_2\right)\fro{Z}, %
\end{align*}
which entails \eqref{eq:ineq-hess-lrqp}: 
\begin{align*}
  \lambda_{f}(Z) 
    &\geq \sup\limits_{Z\in\tansp[\xstar]\mrkk}
    \left(1-\brip-(C_1+C_2)\fro{Z}\right)=1-\brip, 
\end{align*}
where the supremum of the right-hand side is obtained with $W\in\tansp[\mstar]\mrkk$ in a certain direction, such that $\fro{W}\to 0$. 
\end{proof}

\begin{proof}[{\bf Proof of \cref{thm:main-generic}.}] 
Under the $(\brip,2k)$-RPD property, 
the result of \cref{lemm:f-hess-pd} holds, \ie, the Riemannian Hessian of $f$ at $\mstar$ is positive definite; see~\eqref{eq:ineq-hess-lrqp}. 
Therefore, according to~\cite[Theorem 4.5.6]{AbsMahSep2008}, if $\indseq$ converges to $\mstar$, the convergence rate of $\indseq$ is linear.  
\end{proof}

\begin{proof}[{\bf Proof of \cref{lemm:f-lip-muu}.}] 
~
    \begin{itemize}
        \item[(i)]    
The inequality~\eqref{eqs:f-lipschitz-rsc-a} holds since
$f$~\eqref{eq:def-f-lrqp} has Lipschitz-continuous gradient in $\vecsp$, which
is true since it is twice-differentiable and is composed of quadratic terms.
The constant $\lip=\sqrt{1+\brip}$ when $\opa$ is a projection (\eg, \cite{5730578}). 
The existence of constant $\rlip>0$ is dedueced from \cite[Lemma~2.7 and Theorem~2.8]{Boumal2018}. 

\item[(ii)] 
\redl{
Since $f$ is a quadratic function, it is strongly convex on $\reals^{m\times
n}$: there exists $\mu_f>0$ such that for any $X,Y \in \matsp$,
    $$ f(Y) - f(X) \geq \braket{\nabla f(X), Y-X} + \frac{\mu_f}{2} \fro{Y-X}^2.$$
By setting $Y:=\mstar$, the above inequality can be re-written as 
    \begin{align}
        \label{eq:4dot7-pre}
    f(X) - f(\mstar) \leq \braket{\nabla f(X), X-\mstar} - \frac{\mu_f}{2} \fro{X-\mstar}^2.
    \end{align}
Next, we prove an inequality similar to the above result but with the term
$\nabla f(X)$ replaced by the Riemannian gradient $\grad f(X)$. The right
hand-side (RHS) of \eqref{eq:4dot7-pre} can be written as  
    \begin{align*}
     \text{RHS} 
     & = \braket{\grad f(X), X- \mstar} - \frac{\mu_f}{2} \fro{X-\mstar}^2 -
         \braket{\grad f(X)- \nabla f(X), X-\mstar } \\ 
     & = \braket{\grad f(X), X-\mstar } - \frac{\mu_f}{2} \fro{X-\mstar}^2 -
     \underbrace{ \braket{(\opvx-\id)(\nabla f(X)), X- \mstar} }_{\alpha_1}, 
    \end{align*}
    where the third term above with $\opvx(\cdot)$ is a result of~\eqref{eq:rgrad-as-pegrad} in \Cref{prop:rgrad-total2quot}. 

It remains to prove that $\alpha_1 \geq 0$ under the $\dnm_0$-closeness condition: the term $\alpha_1$ satisfies the following property, given that $\nabla f(X) = \opa(X-\mstar)$, 
\begin{align*}
        \alpha_1 
        & = \braket{ (\opvx-\id)(\opa(X-\mstar)), X-\mstar } 
        = \braket{ \opa(X-\mstar), (\opvx-\id)(X-\mstar) } \\
        & = \braket{ \opa(X-\mstar), (\opvx-2\id)(X-\mstar) } + \braket{\opa(X-\mstar), X-\mstar} \\
        & = 2 \braket{ \opa(X-\mstar), (I-\proj_{\tansp[X]\mrkk})(\mstar)} + \braket{\opa(X-\mstar), X-\mstar}, 
\end{align*}
where the last equality is due to the property of the operator $\opvx$ (\Cref{lemm:opvx-props}) that
        $ (\opvx-2I)(X-\mstar) = 2(I-\proj_{\tansp[X]\mrkk})(\mstar)$. 
We proceed by evoking the following lemma about the Frobenius norm of $(I-\proj_{\tansp[X]\mrkk})(\mstar)$: 
\begin{lemma}[{\cite[Lemma 4.1]{wei2016guarantees}}]
\label{lemm:wei2016-lemm4.1}
Let $X$ and $Y$ be two matrices in $\mrkk$. % 
Then it holds that 
\begin{equation}
    \label{eq:lemm41-wei2016}
    \fro{(\id-\proj_{\tansp[X]\mrkk})(Y)} \leq \frac{1}{\sigma_k(Y)}\frob[Y-X]^2.
\end{equation}
\end{lemma}
%% ---
Consequently, $\alpha_1$ is lower-bounded as follows: 
\begin{align*}
        \alpha_1 
        & = 2 \braket{ \opa(X-\mstar), (I-\proj_{\tansp[X]\mrkk})(\mstar)} +
        \braket{\opa(X-\mstar), X-\mstar} \\ 
        & \geq -2 |\braket{ \opa(X-\mstar), (I-\proj_{\tansp[X]\mrkk})(\mstar)}| + (1-\beta)\fro{X-\mstar}^2 \\
        & \geq -2\fro{\opa(X-\mstar)}\fro{ (I-\proj_{\tansp[X]\mrkk})(\mstar) } + (1-\beta)\fro{X-\mstar}^2 \\ 
        & \geq -\frac{2 \lip}{\sigmin} \fro{X-\mstar}^3 + (1-\beta)\fro{X-\mstar}^2, 
\end{align*}
where the $(1-\brip)$ term in the first inequality is due to the RPD property, and the first term in the last inequality is a result of (a) Lipschitz-smoothness of $f$ (\Cref{lemm:f-lip-muu} (i)), \ie, $\fro{\opa(X-\mstar)} = \fro{\nabla f(X)-\nabla f(\mstar)} \leq \lip \fro{X-\mstar}$; and (b) the upper-bound~\eqref{eq:lemm41-wei2016}. 
Therefore, we have 
$$ \alpha_1 \geq (1-\beta - \frac{2\lip}{\sigmin} \fro{X-\mstar}) \fro{X-\mstar}^2 \geq 0$$
for any $X$ satisfying $\fro{X-\mstar} \leq \frac{1}{2} \cbd$. 
}
\end{itemize}

\end{proof}

\begin{proof}[{\bf Proof of \cref{lemm:f-rsc-riem}.}]
\label{parag:prf-rsc-riem}
We prove the inequality~\eqref{eq:rgrad-nmb} as follows. First, through \cref{prop:rgrad-total2quot}, the Riemannian gradient of $f$~\eqref{eq:def-f-lrqp} is 
\begin{equation}
    \rgrad = \opvx(\nabla f(X)) =  \opvx(\opa(X-\mstar)),
\end{equation}
for all $X\in\mrkk$. Therefore, we have 
\begin{align}
    &\braket{\rgrad[X], X-\xstar}=\braket{\opvx(\opa(X-\mstar)),X-\mstar}\notag\\
    & =\braket{\opa(X-\mstar),\opvx(X-\mstar)} \label{eq:prf5410-l1}\\
    & = 2\braket{\opa(X-\mstar),X-\mstar} + \braket{\opa(X-\mstar),(\opvx-2\id)(X-\mstar)}\notag\\%
    & = 2\braket{\opa(X-\mstar),X-\mstar} + 2\braket{\opa(X-\mstar),(\id-\proj_{\tansp[X]\mrkk})(\mstar)}\label{eq:prf5410-l3}\\
    &\geq 2\underbrace{\braket{\opa(X-\mstar),X-\mstar}}_{a_1} - 2\underbrace{\fro{\opa(X-\mstar)}}_{a_2}\underbrace{\fro{(\id-\proj_{\tansp[X]\mrkk})(\mstar)}}_{a_3},\label{eq:prf5410-ineq1}
\end{align}
where~\eqref{eq:prf5410-l1} holds since $\opvx$ is a symmetric operator,~\eqref{eq:prf5410-l3} is obtained by noticing that $(\opvx-2\id)(X-\mstar)=2(\id-\proj_{\tansp[X]\mrkk})(\mstar)$; see~\eqref{eq:opvx-iddiff} in \cref{lemm:opvx-props}. 
The terms $a_{1}$ and $a_2$ in~\eqref{eq:prf5410-ineq1} have the following bounds, 
\begin{align}
    a_1&:= \braket{\opa(X-\mstar),X-\mstar} \geq (1-\brip)\frobb[X-\mstar], \label{eq:ineq-a1-ub}\\
    a_2&:= \fro{\opa(X-\mstar)}  \leq \lip\frob[X-\mstar],\label{eq:ineq-a2-lb} 
\end{align}
where~\eqref{eq:ineq-a1-ub} is the result of
\Cref{prop-rpd2asp},
and~\eqref{eq:ineq-a2-lb} holds according to \cref{lemm:f-lip-muu} (recall that $\opa(X-\mstar)=\nabla f(X)=\nabla f(X)-\nabla f(\mstar)$). 

From~\eqref{eq:lemm41-wei2016} of \cref{lemm:wei2016-lemm4.1}, the term $a_3$ in~\eqref{eq:prf5410-ineq1} has the following bound, 
\begin{align}
    a_3:= \fro{(\id-\proj_{\tansp[X]\mrkk})(\mstar)} \leq \frac{1}{\sigmin}\frobb[X-\mstar],\label{eq:ineq-a3-lb} 
\end{align}
where $\sigmin:=\sigma_k(\mstar)$. Applying~\eqref{eq:ineq-a1-ub}--\eqref{eq:ineq-a3-lb} to~\eqref{eq:prf5410-ineq1}, we have
\begin{align}\label{eq:prf5410-ineq2}
    \braket{\rgrad[X], X-\xstar}&\geq 2\Big((1-\brip)- \frac{\lip}{\sigmin}\fro{X-\mstar}\Big)\frobb[X-\mstar].  
\end{align}
Note that the coefficient in the right-hand side of~\eqref{eq:prf5410-ineq2} is strictly positive if 
$X\in\mrkk$ satisfies $\fro{X-\mstar} \leq\dnm$, with $\dnm$ satisfying 
$0\leq\dnm < \cbd$. 
\end{proof}

\begin{proof}[{\bf Proof of \cref{lemm:4.8x}}] 
\label{prf:lemm-4.8x}
    First, note that $\braket{\nabla f(X), \grad f(X)} = \frac{1}{2} \fro{ \grad f(X) }^2$. Indeed, from Lemma x, we have $\opvx\circ \opvx(Z) = 2\opvx(Z)$. Consequently,  
    \begin{align*}
        \braket{g_X, 2\nabla f(X)} 
        & = \braket{g_X, (2I-\opvx + \opvx)(\nabla f(X))} 
        = \fro{g_X}^2 + \braket{g_X, (2I-\opvx)\nabla f(X)} \\
        & = \fro{g_X}^2 + \braket{(2I-\opvx)\circ\opvx(\nabla f(X)), \nabla f(X)} 
        = \fro{g_X}^2,
    \end{align*} 
    where the last line holds because $(2I-\opvx)\circ\opvx(Z) =0$ for any $Z$. 

(i) Consider a decomposition into two orthogonal components $\nabla f(X) := \proj_{\tansp[X]\mrkk}(\nabla f(X)) + P_{\perp}(\nabla f(X))$. Given that $\opvx(Z) = 2Z$ if $Z\in \tansp[X]\mrkk$ and $\opvx(Z) = 0$ if $Z \in (\tansp[X]\mrkk)^{\perp}$, we have 
    \begin{align*}
        \frac{1}{2}\fro{g_X}^2 
        & = \braket{\nabla f(X), g_X} = \braket{\nabla f(X), \opvx(\nabla f(X)) } = 2\braket{\nabla f(X), \proj_{\tansp[X]\mrkk}(\nabla f(X))} \\ 
        & = 2 \fro{\proj_{\tansp[X]\mrkk}(\nabla f(X))}^2. 
    \end{align*} 
    The set $C_{\delta,\bar{\sigma}}^\star=\nmb\cap\{X: \sigma_k(X)\geq \bar{\sigma}\}$ is a compact subset of $\nmb$. Moreover, through \Cref{lemm:f-rsc-riem}, all $X\in C_{\delta,\bar{\sigma}}$ different than $\mstar$ are non-stationary, \ie, $\fro{g_X}\neq 0$. Therefore, there exists a ratio $0<\cca \leq 1$ such that for all $X \in C_{\delta,\bar{\sigma}}^\star$,    
    \begin{align} \label{eq:gx-nablafx}
        \fro{g_X} = 2 \fro{\proj_{\tansp[X]\mrkk}(\nabla f(X))} \geq 2\cca \fro{\nabla f(X)}.
    \end{align}

(ii) When $\opa$ satisfies the RPD property, there exists $\lambda_2\in (0,1)$ such that 
$$\fro{\nabla f(X)}^2=\fro{\opa(X-\mstar)}^2 = \braket{X-\mstar, \opa^2(X-\mstar)} \geq \lambda_2 \fro{X-\mstar}^2$$ since the linear operator $\opa^2$ is symmetric and positive definite. This entails~\eqref{eq:lemm-4.8x} with $\tcca=\cca \sqrt{\lambda_2}$. 
 
For \eqref{eq:lemm-4.8yy}: the term $\Gamma_X= \nabla f(X) X^{\dagger} \nabla f(X)$ satisfies 
    \begin{align*} 
        \fro{ \Gamma_X } 
        & \leq \frac{1}{\sigma_{k}(X)} \fro{ \nabla f(X)}^2  
        \leq \frac{\lip}{\sigma_{\rkval}(X)} \fro{\nabla f(X)}\fro{X-\mstar} 
        \leq \frac{\lip}{2\cca\sigma_k(X)} \fro{g_X} \fro{X-\mstar}, 
    \end{align*}
    where the first inequality is obtained since the operator norm of $X^{\dagger}$ is $\frac{1}{\sigma_{\rkval}(X)}$, and the second and third inequalities are obtained from the 
    $\lip$-Lipschitz continuity of $\nabla f$ and \eqref{eq:gx-nablafx} respecitvely. 
    Then~\eqref{eq:lemm-4.8yy} follows given that $0 < \bar{\sigma}  \leq \sigma_k(X)$. 
\end{proof}

\begin{proof}[{\bf Proof of \Cref{lemm:4.8z}.}]
    \label{prf:lemm-4.8z}
    Throughout the proof, we use the short notation $g_X:= \grad f(X)$ for the Riemannian gradient~\eqref{eq:rgrad-as-pegrad}.  

(i) To verify~\eqref{eq:cond-1}: we observe that for a stepsize $0<\theta\leq
\frac{1}{\rlip}$ %
\begin{align} 
    - D(\theta)
    & =  \frac{\theta}{2} \Big( (2-\rlip\theta)\fro{g_X}^2 -
    2(1-\rlip\theta)\braket{g_X, \theta\Gamma_X} - \rlip\theta \fro{\theta\Gamma_X}^2 \Big) \nonumber \\ 
    & =  \frac{\theta}{2} \Big( 2\fro{g_X}^2 - (\rlip\fro{g_X}^2 +
    2\braket{g_X,\Gamma_X})\theta - (\rlip\theta \fro{\Gamma_X}^2 - 2\rlip \braket{g_X,\Gamma_X} )\theta^2 \Big) \nonumber \\ 
    & \geq \frac{\theta}{2} \big( \underbrace{ 2\fro{g_X}^2 - (\rlip\fro{g_X}^2 +
    2\braket{g_X,\Gamma_X})\theta - (\fro{\Gamma_X}^2 - 2\rlip
    \braket{g_X,\Gamma_X} )\theta^2 }_{:=\fro{g_X}\fro{\Gamma_X}\cdot P_1(\theta)}\big)
\label{eq:ineq-p1}
\end{align}
where the inequality holds due to $-\rlip\theta\fro{\theta\Gamma_X}^2 \geq
-\fro{\theta\Gamma_X}^2$ given that $0<\rlip\theta \leq 1$. 
In~\eqref{eq:ineq-p1}, $P_1(\theta)$ is a quadratic polynomial with the following expression: 
\begin{align} 
    P_1(\theta) & = \frac{1}{\fro{g_X}\fro{\Gamma_X}} \big(2\fro{g_X}^2 - (\rlip\fro{g_X}^2 + 2\braket{g_X,\Gamma_X})\theta - (\fro{\Gamma_X}^2 - 2\rlip \braket{g_X,\Gamma_X} )\theta^2\big) \nonumber\\ 
& = \frac{2}{\gamma_X} - \left(\frac{\rlip}{\gamma_X} + 2\cos(\alpha_X)\right)\theta - (\gamma_X-2\rlip \cos(\alpha_X))\theta^2, 
\label{eq:p1-exp}
\end{align}
where $\gamma_X:=\frac{\fro{\Gamma_X}}{\fro{g_X}} > 0$ and $\cos(\alpha_X):=
\frac{\braket{g_X,\Gamma_X}}{\fro{g_X}\fro{\Gamma_X}}$. 
From~\eqref{eq:ineq-p1}, we have $-D(\theta) \geq 0$ if $P_1(\theta)\geq 0$. 

Now, we show that $P_1(\theta)\geq 0$ for all $0<\theta\leq
\bar{\theta}:=\frac{1}{\rlip}$ if $P_1(\bar{\theta}) \geq 0$. First, notice that $P_1(0) =
\frac{2}{\gamma_X} > 0$. Then, it suffices to have $P_1(\bar{\theta})\geq 0$. This is true for all types of quadratic polynomials with $P_1(0)>0$; in 
the only unobvious case when this quadratic polynomial is decreasing then
increasing on $[0, +\infty)$, it suffices to verify that $\bar{\theta}$ is
smaller than the point $\theta^*$ where $P_1$ attains the minimum, \ie, it suffices to
verify that 
$$\bar{\theta}=\frac{1}{\rlip} \leq \theta^*=\frac{1}{2}\cdot\frac{\frac{\rlip}{\gamma_X} + 2
\cos(\alpha_X)}{2\rlip\cos(\alpha_X)-\gamma_X}.$$ % 
The required inequality above is equivalent to
$\frac{\rlip}{\gamma_X} + 2 \frac{\gamma_X}{\rlip} \geq 2 \cos(\alpha_X)$, which is
always true because $\frac{\rlip}{\gamma_X} + 2 \frac{\gamma_X}{\rlip}  \geq 2
\sqrt{\frac{\rlip}{\gamma_X} \cdot 2 \frac{\gamma_X}{\rlip}} = 2\sqrt{2}$ while the
right hand-side is bounded by $2$. 

Note that the condition $P_1(\bar{\theta})\geq 0$ reads 
\begin{align*} 
P_1(\bar{\theta}) 
& =\frac{1}{\rlip}\left(\frac{2\rlip}{\gamma_X} - (\frac{\rlip}{\gamma_X} + 2\cos(\alpha_X))- (\frac{\gamma_X}{\rlip}-2\cos(\alpha_X))\right)  %
 = \frac{1}{\rlip}\left( \frac{\rlip}{\gamma_X} - \frac{\gamma_X}{\rlip} \right)\geq 0 
\end{align*}
where $\rlip>0$ and $\gamma_X>0$, 
which holds if and only if $X$ satisfies 
\begin{align}
\label{eq:p1-cond}
\gamma_X \leq \rlip, \text{~\ie,~} \fro{\Gamma_X} \leq \rlip \fro{g_X},
\end{align}
hence the sufficient condition for \eqref{eq:cond-1} with any stepsize $0<\theta\leq \frac{1}{\rlip}$. 

(ii) To prove~\eqref{eq:cond-rho}: We deduce an additional sufficient condition 
under the condition~\eqref{eq:p1-cond} obtained in (i). % 

By multiplying both sides of the inequality above with $\rho\geq 1$ and adding $-\tilde{D}(\theta)$, we have %
\begin{align} 
    -\rho D(\theta)-\tilde{D}(\theta) & \geq \frac{\theta}{2}\Big( (2\rho-1)\fro{g_X}^2 - 2\braket{\Gamma_X,X-\mstar} - \big(\rho \rlip\fro{g_X}^2 + 2(\rho-1)\braket{g_X,\Gamma_X}\big)\theta  \nonumber \\ 
    & \quad\qquad -\big((\rho+1) \fro{\Gamma_X}^2 - 2\rho \rlip \braket{g_X,\Gamma_X} \big)\theta^2\Big) 
    ~=: \frac{\theta}{2} \fro{g_X} \fro{\Gamma_X} P_{\rho}(\theta), \label{eq:prho-def}
\end{align}
where %
the $P_{\rho}(\theta)$ can be expressed in $\gamma_X$ and $\cos(\alpha_X)$ as follows: 
\begin{align} 
P_{\rho}(\theta) 
& =  \frac{2\rho-1}{\gamma_X} - 2\tilde{I}_X - (\frac{\rho\rlip}{\gamma_X} + 2(\rho-1)\cos(\alpha_X))\theta - ((\rho+1)\gamma_X-2\rho\rlip \cos(\alpha_X))\theta^2, 
    \label{eq:prho-exp}
\end{align}
where 
$\tilde{I}_X:= \frac{\braket{\Gamma_X,X-\mstar}}{\fro{g_X}\fro{\Gamma_X}}$. 
For~\eqref{eq:prho-exp} to be positive, first let the zero-th order term be strictly positive, similar to $P_1(\theta)$, 
which requires $P_{\rho}(0)=\frac{2\rho-1-2\tilde{I}_X\gamma_X}{\gamma_X} >0$. 
Note that for any $X$ satisfying~\eqref{eq:p1-cond}, %
\ie, $X\in \{X\in\mrkk: \fro{\Gamma_X} \leq \rlip \fro{\grad f(X)}\}$, we have  
\begin{align} 
    \gamma_X P_{\rho}(0) + 1 & = 
    2\rho -2\tilde{I}_X\gamma_X  = 2\rho - \frac{2\braket{\Gamma_X, X-\mstar}}{\fro{g_X}^2} \nonumber \\
    & \geq 2\rho - \frac{2\fro{\Gamma_X} \fro{X-\mstar}}{\fro{g_X}^2}  
    \underset{\eqref{eq:p1-cond}}{\geq} 2\rho - 2\rlip  \frac{\fro{X-\mstar}}{\fro{g_X}}  \nonumber \\ %
    & \geq 2\rho - \frac{\rlip}{\tcca}, \label{eq:ix-2}  
\end{align} 
where~\eqref{eq:ix-2} %
is obtained from % 
{\Cref{lemm:4.8x}} for a constant $0<\tcca \leq 1$. % 
Therefore, we will have $P_{\rho}(0)\geq \frac{1}{\gamma_X}$ (same order of magnitude as $P_1(0)$) if %
$ 2\rho -\frac{\rlip}{\tcca} -1 \geq 1$, \ie, if   
\begin{align}
    \label{eq:rho-range}
    \rho \geq 1 + \frac{\rlip}{2\tcca} \asymp O(1+\rlip).
\end{align}
Subsequently, given $\rho$ as in~\eqref{eq:rho-range}, we use the same technique for (i) to show that  
$P_{\rho}(\theta) \geq 0$ for all $0 < \theta \leq \bar{\theta}_{\rho}:= (1-\frac{1}{\rho})\frac{1}{\rlip}$, if $P_{\rho} (\bar{\theta}_{\rho}) \geq 0$. 
Indeed, this can be verified for all types of quadratic polynomials with $P_{\rho}(0) >0$; in the only unobvious case when $P_{\rho}(\theta)$ is decreasing then increasing on $[0,\infty)$, it holds that for any $\rho >1$, $\bar{\theta}_{\rho} \leq \theta^*_{\rho}$, where $\theta^*_{\rho}$ is the minimizer of $P_{\rho}$. Hence, it always suffices to require $P_{\rho}(\bar{\theta}_{\rho}) \geq 0$. Note that we have 
\begin{align*}
    P_{\rho} (\bar{\theta}_{\rho}) = \frac{1}{\rlip} \left((\rho - 2\tilde{I}_X\gamma_X)\frac{\rlip}{\gamma_X} - \frac{(\rho-1)(\rho^2-1)}{\rho^2}\frac{\gamma_X}{\rlip}\right) \geq  
\frac{1}{\rlip} \left(\frac{\rlip}{\gamma_X} - \frac{(\rho-1)(\rho^2-1)}{\rho^2}\frac{\gamma_X}{\rlip}\right), 
\end{align*} 
which means 
$P_{\rho} (\bar{\theta}_{\rho}) \geq 0$ if $\frac{1}{\rlip} \left(\frac{\rlip}{\gamma_X} - \frac{(\rho-1)(\rho^2-1)}{\rho^2}\frac{\gamma_X}{\rlip}\right)\geq 0$, 
hence the sufficient condition 
\begin{align}
\gamma_X \leq \rlip \frac{\rho}{\sqrt{(\rho-1)(\rho^2-1)}}
    \label{eq:prho-cond}
\end{align} 
where $\rho$ is given in \eqref{eq:rho-range}. 
By combining~\eqref{eq:p1-cond} and \eqref{eq:prho-cond}, we have the sufficient condition~\eqref{eq:p-2cond} on $X$. 
\end{proof}

\begin{proof}[{\bf Proof of \Cref{coro:4.8z}.}]
    For any $X\in C_{\delta_0,\bar{\sigma}}$, the inequality~\eqref{eq:lemm-4.8yy} in \Cref{lemm:4.8x} entails that
$\frac{\fro{\Gamma_X}}{\fro{g_X}}\leq \frac{\lip}{2\cca\bar{\sigma}} \fro{X-\mstar}$, and therefore the condition \eqref{eq:p-2cond} can be satisfied if $\frac{\lip}{2\cca\bar{\sigma}} \fro{X-\mstar} \leq C_{\rho} \rlip$, \ie, if  
    $$ \fro{X-\mstar} \leq \frac{2\cca C_{\rho}\rlip}{\lip} \bar{\sigma}:= \ccb \delta_0$$
    where 
$\ccb= \frac{\rlip\bar{\sigma}}{\sigmin} \frac{2\cca C_{\rho}}{1-\brip}$. 
\end{proof}

\begin{proof}[{\bf Proof of \cref{lemm:f-rsc-fort}.}] 
We start by proving that $\{X_t\}_{t\geq 0}$ converges to $\mstar$ by checking  
initial conditions including the $\delta$-closeness conditions of \Cref{coro:4.8z}. 

First, we show that $\{X_t\}_{t\geq 0}$ by Algorithm 1 
is included in a compact subset of $\mrkk$, in view of the set $C_{\delta,\bar{\sigma}}^\star$ in \Cref{lemm:4.8z}--\Cref{coro:4.8z}. 
It suffices to ensure that the sequence $\{X_t\}$ by Algorithm 1 does not get arbitrarily close to any point in $\mleqk[(k-1)]$. Note that (i) for any $Y \in \mleqk[(k-1)]$, $\fro{Y-\mstar} \geq \sigma_k^\star$, and (ii) $\{X_t\}_{t\geq 0}$ is monotonically decreasing in $f$ values, hence it suffices for $X_0$ to satisfy 
\begin{align*} 
    f(X_0)-f(\mstar) \leq \min_{Y\in \mleqk[(k-1)]} \{f(Y) - f(\mstar)\}, 
\end{align*} 
which, under the $(\brip,2k)$-RPD property, can be satisfied if $(1+\brip)\fro{X_0-\mstar}^2 \leq (1-\brip) \min_{Y\in \mleqk[(k-1)]}\fro{Y-\mstar}^2$, \ie, if $\fro{X_0-\mstar}  \leq \delta_1:=\sqrt{\frac{1-\brip}{1+\brip}} \sigmin$. % 
Therefore, under this $\delta_1$-closeness condition, 
$\{X_t\}_{t\geq 0}$ is closed, \ie, there exists $\bar{\sigma}>0$ such that $\sigma_k(X_t)\geq \bar{\sigma}$ for all $t\geq 0$. Consequently $\{f(X_t)\}_{t\geq 0}$ is closed given that $f$ is continuous. 
Then it follows that $\{f(X_t)\}_{t\geq 0}$ converges, since $\{f(X_{t}) \}_{t\geq 0}$  is monotonically decreasing by Algorithm~1. 
 
Furthermore, through the RPD property, $\{f(X_t)\}_{t\geq 0}$ being decreasing and convergent entails that 
(i) $\{X_t\}_{t\geq 0}$ converges and 
(ii) the sequence %
$\{\max_{s\geq t} \fro{X_s -\mstar}\}_{t\geq 0}$ is decreasing. 
Consequently, when $\fro{X_0-\mstar}\leq \min(\delta_1,\bar{\delta}_0)$ for $\bar{\delta}_0$ given in \Cref{coro:4.8z}, the limit point $X^*:=\lim_{t\geq 0} \{X_t\}$ is included in $C_{\bar{\delta}_0,\bar{\sigma}}^\star$~\eqref{def:set-c}. 
Finally, we show that the limit point $X^*$ is $\mstar$. %
Suppose that $X^* \in  C_{\bar{\delta}_0,\bar{\sigma}}^\star$ is not $\mstar$, then according to  
\Cref{coro:4.8z}, there exists a stepsize $0<\theta \leq (1-\frac{1}{\rho})\frac{1}{\rlip}$ such that Algorithm 1 admits a strict descent from $X^*$ in the $f$ value, %
hence a contradiction. 
Therefore, in conclusion, $\{X_t\}_{t\geq 0}$ converges to $\mstar$, given that $\fro{X_0-\mstar} \leq \bar{\delta}_1:=\min(\bar{\delta}_0, \delta_1)$, where $\bar{\delta}_0 \asymp \delta_0$ and $\delta_1 \asymp \delta_0$. 

Moreover, since $X_0$ satisfy the proximity conditions required in \Cref{coro:4.8z}, $(X_0, X_1)$ satisfies \eqref{ineq:tar}, which, combined with the RPD property, entails that  
$\fro{X_1-\mstar} \leq (1+\theta_0\Delta_{f,k}) \fro{X_0 - \mstar}$, where $-1<\Delta_{f,k}<2\rho -1$ is a constant depending on $(\brip,\rlip, \mu)$. Hence, $\fro{X_1-\mstar}\leq \bar{\delta}_0$ holds if $X_0$ satisfies $\fro{X_0 - \mstar} \leq \frac{1}{1+\Delta_{f,k}} \bar{\delta}_0 \asymp \bar{\delta}_0$. In conclusion, there exists a constant $\ccc \asymp 1$ such that, when $\fro{X_0-\mstar}\leq \min(\delta_0,\ccc\delta_0)$, the two descent conditions \eqref{eq:cond-1}--\eqref{eq:cond-rho} hold for all iterates, \ie, 
\begin{align}
    \label{ineq:tar-b} 
    R_{s+1} \leq \big(1 -\frac{1}{\rho}\big) R_s + \frac{1}{\rho} \big(\frac{1-\theta_0\muu}{2\theta_0} \fro{X_{s} - \mstar}^2  - \frac{1}{2\theta_0} \fro{X_{s+1} - \mstar}^2\big) \quad\forall ~s\geq 0,
\end{align}
where $\theta_0=(1-\frac{1}{\rho})\frac{1}{\rlip}$ is in the range of valid stepsizes given by \Cref{coro:4.8z}. 
Finally, for a strictly positive constant $\epsilon:=\min(\frac{1}{\rho}, \mu\bar{\theta}_0)<1$, we obtain the following inequality by summing both sides of \eqref{ineq:tar-b} times $(1-\epsilon)^{-s}$ for $s\in \{0,1,\dots, t\}$,  
$$
(1-\epsilon)^{-t} R_{t+1} \leq \Big(1-\frac{1}{\rho}\Big) R_0 + \frac{1}{\rho}\Big(\frac{1-\muu\theta}{2\theta} \fro{X_0 - \mstar}^2\Big), 
$$
where $(1-\frac{1}{\rho}) \leq  (1-\epsilon)$. 
Note that $R_0 = \braket{\opa(X_0-\mstar) , X_0 - \mstar} \leq (1+\brip) \fro{X_0-\mstar}^2$, hence we have  
$f(X_{t+1}) - f(\mstar) \leq (1-\epsilon)^{t+1}  C_{f,k}\fro{X_0 -\mstar}^2$  
for a constant $C_{f,k}>0$ depending only on $(\brip, \rlip, \muu)$. 
\end{proof}

\subsection{Proofs for~\Cref{ssec-ropt-mk:appl-mc}}

\begin{lemma} %
    \label{lemm:rinfn}
Suppose that $X,\mstar\in\mrkk$ are sufficiently incoherent: 
    $$\max(\sqrt{m}\|Y\|_{2,\infty}, \sqrt{n}\|\trs[Y]\|_{2,\infty}) \leq B \quad \text{for~} Y \in \{X, \mstar\},$$ and 
that $X$ satisfies $\fro{X-\mstar} \leq \epsilon_0 \sigma_k(\mstar)$, 
{and $\rinfnorm[X-\mstar]\leq \epsilon_0 \rinfnorm[\mstar]$} 
for $\epsilon_0>0$. %
Then the RGD update $\tilde{X} := X - \theta (\grad f(X) - \theta \Gamma_X)$ satisfies 
$$\max(\sqrt{m}\|\tilde{X}\|_{2,\infty}, \sqrt{n}\|\trs[{\tilde{X}}]\|_{2,\infty}) \leq (1+\crinf)B$$ % 
where the parameter 
$\crinf :=  (1+\brk)\theta\epsilon_0 + \brk\theta^2\epsilon_0^2$ with $\brk\geq\sigma_{k}^{-1}(X)\sigmin>0$. 
\end{lemma}
\begin{remark}
    \normalfont 
    The upper bound in this lemma will be an over-estimation if the RGD update with stepsize $\theta$ ensures an effective decrease in the $(2,\infty)$-norm. However, we are content with this bound in order not to add complication (or interfere with) to the contraction results in the Frobenius norm in 
\Cref{ssec-ropt-mk:convana}. 
\end{remark}

\begin{proof}
    First, %
    the gradient term is an $m\times n$ matrix sufficiently incoherent: $\grad f(X)=\opvx(\nabla f(X))=\pu \po(Z) + \po(Z)\pu[V]$ where $Z:=X-\mstar$ and $(U,V)$ are the unitary matrices in the $\rkval$-SVD $X:=U\Sigma\trs[V]$. Hence for every $i$ we have:  
    \begin{align}
        \|\rowof{(\pu \po(Z))}{i}\|_2 
        & = \|\rowof{(U\Sigma)\Sigma^{-1}\trs[U]\po(Z))}{i}\|_2  \nonumber \\ 
        & \leq  \|\rowof{(U\Sigma)}{i} \|_2 \nop{\Sigma^{-1}\trs[U]\po(Z)}  
          = \|\rowof{X}{i}\|_2 \nop{X^\dagger\po(Z)}, \label{eq:pinvpo}  \\ 
        & \leq \sigma_{k}^{-1}(X)\fro{Z}\rinfnorm[X], 
           \label{eq:pupo-i2} 
    \end{align}
    where the equality in~\eqref{eq:pinvpo} is due to $\|\rowof{(U\Sigma)}{i}\|_2 = \|\rowof{X}{i}\|_2$ and $\|\Sigma^{-1}\trs[U]\po(Z) y\|_2 = \|V (\Sigma^{-1}\trs[U]\po(Z) y)\|_2$ for any $y\in\reals^n$ (since $\trs[V]V=I_{\rkval\times \rkval}$), 
    and inequality~\eqref{eq:pupo-i2} is obtained after $\nop{\po(Z)}\leq \fro{\po(Z)} 
    \leq\fro{Z}$. 
    On the other hand, it holds that 
    \begin{align}
           \label{eq:popv-i2} 
        \|\rowof{(\po(Z)\pu[V])}{i}\|_2 \leq \|\rowof{(\po(Z))}{i}\|_2\leq \|\rowof{Z}{i}\|_2,
    \end{align}
    where the first inequality is because $\pu[V]$ is an orthogonal projection onto a strict subspace of $\reals^n$ (with dimension $\rkval < n$). 
    By combining~\eqref{eq:pupo-i2} and~\eqref{eq:popv-i2}, %
    along with the two $\epsilon_0$-proximity bounds of $X$ (in terms of $Z=X-\mstar$), we have $\|\rowof{(\grad f(X))}{i} \|_2 \leq {(1+\brk)\epsilon_0\frac{B}{\sqrt{m}}}$. %

    Second, by applying the same techniques, we have 
    \begin{align} 
        \label{eq:rinfn-gamma}
        \|\rowof{(\Gamma_X)}{i}\|_2 
        \leq \|\rowof{(\po(Z))}{i} \|_{2} \nop{X^{\dagger}\po(Z)} \leq \|\rowof{Z}{i}\|_2 \nop{X^\dagger\po(Z)} 
        \leq \|\rowof{Z}{i}\|_2  (\sigma_k^{-1}(X)\fro{Z}). 
    \end{align}
    Given the $\epsilon_0$-proximity bounds,~\eqref{eq:rinfn-gamma} entails that 
        $\|\rowof{(\Gamma_X)}{i}\|_2  \leq 
                \rinfnorm[X-\mstar]  (\sigma_k^{-1}(X)\fro{X-\mstar}) \leq 
                \brk\epsilon_0^2 \frac{B}{\sqrt{m}}$. 
 
    Finally, we have 
    $\sqrt{m}\|\rowof{\tilde{X}}{i}\|_2 \leq \Big(1+ \theta (1+\brk)\epsilon_0 + \theta^2\brk\epsilon_0^2\Big) B$ for all $i$. 
    Hence the parameter $\crinf= \theta (1+\brk)\epsilon_0 + \theta^2\brk\epsilon_0^2$. 
    The same conclusion applies to $\sqrt{n}\rinfnorm[{\trs[{\tilde{X}}]}]$ by the transpose operation and exchange of roles of $U$ and $V$, noticing that $\po$ (and $\opa$ in general) is a symmetric operator.  
\end{proof}

\begin{proof}[{\bf Proof of \Cref{prop:mc-rip}.}]
    We show that the assumptions of this propositon can be converted into
    the assumpptions of \cite[Theorem 3.1, Claim 3.1, Lemma~3.1]{Sun2016a},
    based on the fact that they all describe an incoherence neighborhood of $\mstar$; the 
    only difference is in the matrix representations: in this proposition, we use
    the $m\times n$ matrix representation, while \cite[Claim~3.1]{Sun2016a} uses the matrix factorization representation $(G,H)$ such
    that $G\trs[H] = X\in\reals^{m\times n}$, and the incoherence
    neighborhood is $ K_1 \cap
    K_2 \cap K_{\delta}$ (see~\cite[eq.(30)]{Sun2016a}), where
    $K_{\delta}:=\{(G,H)| \fro{G\trs[H] - \mstar} \leq \delta\}$, and $K_1$ and
    $K_2$ are sets of $(G,H)$ with bounded max row-norms and Frobenius norms respectively.
     
    \bleu{For any $X\in \nmg[{C_B, \mu}]$,} let $(G,H)\in\prodsp$ be a
    pair of matrices satisfying $G\trs[H]=X$ and $\fro{G}=\fro{H}$. Then, it
    follows that (i) $(G,H)\in K_1$,
    because $G$ and $H$ are balanced and share the same set of left and right singular vectors
    with $X$, %
    respectively, hence they are incoherent: $\rinfnorm[G] \leq
    \sqrt{2\cst\sigmax\rkval} \sqrt{\frac{\mu\rkval}{m}}$ and $\rinfnorm[H] \leq
    \sqrt{2\cst\sigmax\rkval} \sqrt{\frac{\mu\rkval}{n}}$; 
    (ii) $(G,H)\in K_2$ (with an uniform bound ``$\beta_T$'' on Frobenius norms) holds with the given choice $\fro{G}=\fro{H}$, provided
    that the bound $\beta_T$ fits $\max_{X\in\nmb[{\delta}]} \fro{X}$; and (iii) $(G,H)\in
    K_{\delta}$, simply because $G\trs[H]=X$ and $\fro{X-\mstar} \leq \delta$. 
    Hence the assumptions of \cite[Lemma
    3.1]{Sun2016a} hold, therefore the conclusion.  
\end{proof}

\begin{proof}[{\bf Proof of \Cref{lemm:mc-proj}.}]
~By construction, $\dproj(\tilde{X})$ has the same incoherence bound~\eqref{eq:sameinc} as a result of the normalization operator $\dproj$ at $\tilde{X}$. We show the non-expansiveness~\eqref{eq:nonexpa} as follows. 

The following property \cite{tong2021accelerating,Zheng2016} %
is used: %
\begin{property}
    \label{claim:lemm19}
For vectors $u, u^{\star}\in\reals^n$ and $\lambda \geq \frac{ \| u^\star \|_2 }{ \| u\|_2 }$, it holds that $\| (1 \wedge \lambda) u - u^\star \|_2 \leq  \| u - u^\star \|_2$. 
\end{property}
In view of \Cref{claim:lemm19}, it holds that, %
for $D^{(1)}_{ii} = \min(1, \lambda_i)$ and $\lambda_i:=\frac{B}{\sqrt{m}\|\rowof{\tilde{X}}{i}\|_2}$, 
\begin{align}
    \label{eq:lemm-proj-1}
\|\rowof{(\dproj(\tilde{X}) - \mstar)}{i}\|_2 = \| D^{(1)}_{ii} \rowof{(\tilde{X} D^{(2)})}{i} - \rowof{\mstar}{i} \|_2 \leq \|\rowof{(\tilde{X} D^{(2)} - \mstar)}{i}\|_2
\end{align}
as long as $\lambda_i=\frac{B}{\sqrt{m}\|\rowof{\tilde{X}}{i}\|_2} \geq \frac{\|\rowof{\mstar}{i}\|_2}{\| \rowof{( \tilde{X}D^{(2)})}{i}\|_2}$, 
which can be satisified  
when $B\geq (1+\crinf)\sigmax\sqrt{\mu\rkval}$. More precisely, we have 
\begin{enumerate}
    \item $\|\rowof{\mstar}{i} \|_2 \leq \|\rowof{U^\star}{i} \Sigma^\star\|_2 \leq 
        \sigmax\rinfnorm[U^\star]  \leq \sigmax \sqrt{\frac{\mu \rkval}{m}}$. 
    \item The ratio $\frac{\|\rowof{ \tilde{X} }{i}\|_2}{\|\rowof{(\tilde{X}D^{(2)})}{i}\|_2}\leq 1+\crinf$. %
       This is because: for every $j$, $(\tilde{X} D^{(2)})_{ij} = D^{(2)}_{jj} \tilde{X}_{ij}$ where $D^{(2)}_{jj} = \min(1, \frac{B}{\sqrt{n}\|\colof{\tilde{X}}{j}\|_2})$ is lower-bounded % 
since $\|\colof{ \tilde{X} }{j}\|_2$ is also small enough. 
Indeed, through {\Cref{lemm:rinfn}} we have 
        $\sqrt{n}\|\colof{\tilde{X}}{j} \|_2 \leq (1+\crinf) B$ for $\crinf$ given in the statement. 
        Hence for all $j$, $ \frac{1}{1+\crinf} \leq D^{(2)}_{jj} \leq 1$, and $\frac{1}{1+\crinf}  |\tilde{X}_{ij}| \leq |(\tilde{X}D^{(2)})_{ij}| \leq  |\tilde{X}_{ij}|$. 
\end{enumerate}
As a consequence of~\eqref{eq:lemm-proj-1} for all $i$, we have %
$\fro{P_B(\tilde{X})  - \mstar} \leq \fro{\tilde{X}D^{(2)}-\mstar}$. 

Finally, using the same argument, we have further that  
$ \fro{ \tilde{X}D^{(2)}  - \mstar} \leq  \fro{\tilde{X}-\mstar}$. 
This is because: for every $j$, it holds that $\|  \rowof{( \tilde{X}D^{(2)} - \mstar)}{j}   \|_2 = \| D^{(2)}_{jj} \colof{\tilde{X}}{j} - \colof{\mstar}{j} \|_2\leq \| \colof{\tilde{X}}{j} -\colof{\mstar}{j}\|_2$ 
for $D^{(2)}_{jj} = \min(1, \frac{B}{\sqrt{n}\|\colof{\tilde{X}}{j}\|_2})$, 
if $\lambda_j:= \frac{B}{\sqrt{n}\|\colof{\tilde{X}}{j}\|_2}  \geq \frac{\|\colof{\mstar}{j}\|_2}{\|\colof{\tilde{X}}{j}\|_2}$, \ie, if $\frac{B}{\sqrt{n}}\geq \|\colof{\mstar}{j}\|_2$, which is true since $B \geq \sqrt{n}\|\trs[\mstar] \|_{2,\infty} \geq \sqrt{n}\|\colof{\mstar}{j} \|_2$ (for the same reason as point 1. above). 
\end{proof} 

\begin{proof}[{\bf Proof of \Cref{thm:main-mc}.}]
    The linear convergence to $\mstar$ is a result of \Cref{lemm:f-rsc-fort}, given \Cref{prop:mc-rip} and \Cref{lemm:mc-proj}.  \Cref{lemm:mc-proj} confirms that the normalized RGD rule preserves the incoherence bound of $X_t$ and is also non-expansive, which 
ensure that $\{X_t\}_{t\geq 0} \subset \nmg[{C_B,\mu}]$~\eqref{def:nm-sl15}.
Hence the RPD property~\eqref{eq:mc-rip} holds for all $X_t$. 
\end{proof}
%%---

%% Bibliography 
% \bibliographystyle{alphaurl}

\end{document}